\documentclass[12pt]{article}

\usepackage{a4}   
\usepackage{amssymb,amsmath,mathrsfs,textcomp}   
\usepackage[amsmath,thmmarks]{ntheorem}   
\usepackage{yhmath}   
\usepackage{indentfirst}   

\theoremstyle{nonumberplain}
\theorembodyfont{\rm}
\theoremseparator{.}

\theoremstyle{plain}
\theorembodyfont{\slshape}
\theoremseparator{.}
\newtheorem{lemma}{Lemma}
\newtheorem{corollary}[lemma]{Corollary}
\newtheorem{theorem}{Theorem}

\theoremstyle{nonumberplain}
\theoremheaderfont{\bf}
\theorembodyfont{\rm} 
\theoremseparator{.}
\theoremsymbol{$\Box$}
\newtheorem{proof}{Proof}

\newcommand{\N}{\mathbb{N}}

\newcommand{\R}{\mathbb{R}}
\newcommand{\C}{\mathbb{C}}
\newcommand{\D}{\mathbb{D}}

\newcommand{\coloneqq}{\mathrel{\mathop:}=}
\newcommand{\eqqcolon}{=\mathrel{\mathop:}}

\providecommand{\abs}[1]{\lvert#1\rvert}
\providecommand{\gabs}[1]{\big\lvert#1\big\rvert}
\providecommand{\Gabs}[1]{\Big\lvert#1\Big\rvert}
\providecommand{\norm}[1]{\lVert#1\rVert}

\DeclareMathOperator{\dist}{dist}
\DeclareMathOperator{\vol}{vol}

\newcommand{\sL}{\mathscr{L}}
\providecommand{\Bbarn}{\overline{B}^{\begin{minipage}{1ex} $\scriptstyle \vspace{-1.2ex} n$ \end{minipage}}}
\providecommand{\Bbarnn}{\overline{B}^{\begin{minipage}{3.4ex} $\scriptstyle \vspace{-1.2ex} n-1$ \end{minipage}}}

\begin{document}
%
%
\title{Wermer type sets and extension of CR functions}
\author{Tobias Harz, Nikolay Shcherbina and Giuseppe Tomassini $^*$}
\maketitle
%
\small\noindent{{\bf Abstract.} For each $n\geq2$ we construct an unbounded closed pseudoconcave complete pluripolar set $\mathcal E$ in $\mathbb C^n$ which contains no analytic variety of positive dimension (we call it a \textit{Wermer type set}). We also construct an unbounded strictly pseudoconvex domain $\Omega$ in $\mathbb C^n$ and a smooth $CR$ function $f$ on $\partial\Omega$ which has a single-valued holomorphic extension exactly to the set $\overline\Omega\setminus\mathcal E$.}
%
%
\section{Introduction}
\renewcommand{\thefootnote}{}\footnote{2010 \textit{Mathematics Subject Classification.} Primary 32D10, 32V10, 32T15; Secondary 32D20, 32V25.}\footnote{\textit{Key words and phrases.} Envelopes of holomorphy, $CR$ functions, strictly pseudoconvex domains, analytic structure.}\footnote{\!\!\!$^*$Supported by the project MURST "Geometric Properties of Real and Complex Manifolds".}
In this paper we are dealing with the extension problem of $CR$ functions defined on the boundary $\partial \Omega$ of an unbounded domain $\Omega$ in $\C^n$, $n\ge 2$. When $\Omega$ is bounded with a connected smooth boundary (no hypothesis of pseudoconvexity), holomorphic extension of $CR$ functions to the whole of $\Omega$ is granted by the classical result of Bochner (see, for example, Theorem 2.3.2$^{'}$ in [H]). In particular, if $\Omega$ is a domain of holomorphy, the envelope of holomorphy $E(\partial\Omega)$ of $\partial\Omega$ (i.e., the envelope of $\partial\Omega$ with respect to the algebra of $CR$ functions on $\partial\Omega$; for details see, for example, [J], [JS], [MP], [St]) coincides with $\overline\Omega$. For unbounded domains, such an extension result is not longer true in general, even for strictly pseudoconvex domains, as shown by the following example.\medskip

\noindent{\bf Example}. Let $f$ be an entire function in $\C^2$ and
\[ \Omega:= \big\{z \in \C^2: \log\abs{f(z)} + C_1\norm{z}^2 < C_2 \big\} \]
where $C_1$ and $C_2$ are constants and $C_1>0$. For almost all constants $C_2$, $\Omega$ is an unbounded strictly pseudoconvex domain with smooth boundary in $\C^2$ containing the divisor $\{f=0\}$. We are going to show that $E(\partial\Omega)$ is one-sheeted, contained in $\Omega$ and
\[ \overline{\Omega}\setminus E(\partial\Omega) = \big\{z \in \C^2: f(z)=0 \big\}. \]
Fix an exhaustion $V_1 \subset\subset V_2 \subset\subset \cdots \subset\subset \partial\Omega$ of $\partial\Omega$ by relatively compact subsets. Intersecting $\Omega$ by balls $B^2(0,R_k) \subset \C^2$ centered at the origin, of radius $R_k$, in such a way that $V_k \subset\subset \partial\Omega\cap B^2(0,R_k)$ and then smoothing the edges as in [To], we can find strictly pseudoconvex bounded domains $\Omega_k$ in $\C^2$ such that $V_k \subset \partial\Omega_k \cap \partial\Omega$ for every $k\in\N$. Let $\Gamma_k := \partial\Omega_k\setminus V_k$. Then, in view of Theorem A from [J], one has 
\[ E(V_k) = E(\partial\Omega_k \setminus \Gamma_k) = \overline{\Omega}_k\setminus \widehat{\Gamma_k}^{\mathcal{A}(\Omega_k)} \subset \C^2, \]
where $\widehat{\Gamma_k}^{\mathcal{A}(\Omega_k)}$ is the $\mathcal{A}(\Omega_k)$-hull of $\Gamma_k$, i.e., the hull of $\Gamma_k$ with respect to the algebra of holomorphic functions on $\Omega_k$ which are continuous up to the boundary. It follows that $E(\partial\Omega) := \bigcup_{k=1}^{\infty}E(V_k)$ is one-sheeted. We just have to show that $\overline{\Omega}\setminus E(\partial\Omega)$ is the divisor $\{f=0\}$. Since the $CR$ function $1/f$ on $\partial\Omega$ does not extend to $\{f=0\}$, it follows that $\{f=0\}\subset\overline\Omega\setminus E(\partial\Omega)$. Conversely, filling $\overline{\Omega}\setminus\{f=0\}$ by the following family of holomorphic curves, $\gamma_w = \{z \in \overline{\Omega} : f(z) = w\}$, where $w \in \C \setminus \{0\}$, and using the Kontinuit\"atssatz, it turns out that $\overline{\Omega}\setminus\{f=0\} \subset E(\partial\Omega)$.\bigskip

In this context we have to mention Tr\'epreau's Theorem [Tr] stating that, given a point $z$ in a smooth hypersurface $M\subset\C^n$, the homomorphism
$$
\mathcal O_z\rightarrow\varinjlim\limits_{U\ni z}\mathcal O(U\setminus M)
$$
is onto if and only if no germ of a complex hypersurface passing through $z$ is contained in $M$. We also recall Chirka's generalization [C] of Tr\'epreau's result (in the case $n=1$ this generalization can also be obtained from the earlier work [Sh]): Let $\Gamma\subset\C^{n+1}$ be a continuous graph over a convex domain $D\subset\C^n\times\R$ and $z\in\Gamma$ be a point such that none of the connected components of $(D \times\R) \backslash\Gamma$ is extendable holomorphically to $z$. Then, $z$ is contained in an $n$-dimensional holomorphic graph lying on and closed in $\Gamma$.\bigskip

A natural question arises: Let $\Omega$ be an unbounded strictly pseudoconvex domain  in $\C^n$, $n \ge 2$, such that $E(\partial\Omega)$ is one-sheeted and $\overline{\Omega}\setminus E(\partial\Omega)\neq\varnothing$; does $\overline{\Omega}\setminus E(\partial\Omega)$ possess an analytic structure? In this paper we prove that the answer to this question is negative. Precisely, we prove the following two theorems.

\begin{theorem} For each $n\in\N$, $n\ge 2$, there exist a closed set $\mathcal E\subset\C^n$ which contains no analytic variety of positive dimension and a plurisubharmonic function $\varphi:\C^n\to[-\infty,+\infty)$ such that
\begin{enumerate}
 \item[(1)] $\mathcal E=\{z\in\C^n:\varphi(z)=-\infty\}$;
 \item[(2)] The function $\varphi$ is pluriharmonic on $\C^n\setminus\mathcal E$;
 \item[(3)] The domain $\C^n\setminus\mathcal E$ is pseudoconvex;
 \item[(4)] For every $R>0$ one has $\widehat{\partial B^n(0,R)\cap\mathcal E} = \Bbarn(0,R)\cap\mathcal E$, where $B^n(0,R) \subset \C^n$ is the ball of radius $R$ centered at the origin and $\widehat{\partial B^n(0,R)\cap\mathcal E}$ denotes the polynomial hull of the set $\partial B^n(0,R)\cap\mathcal E$.
\end{enumerate} 
\end{theorem}

\begin{theorem} For each $n \in \N$, $n \ge 2$, there exist an unbounded strictly pseudoconvex domain $\Omega$ in
$\C^n$ with smooth boundary, a closed subset $\mathcal{E}$ of \,$\C^n$ and a smooth $CR$ function $f$ on $\partial\Omega$ such that
\begin{enumerate}
 \item[(1)] $\mathcal{E}\subset\Omega$, and it contains no analytic variety of positive dimension;
 \item[(2)] $f$ has inside $\Omega$ a single-valued holomorphic extension exactly to $\Omega\setminus\mathcal{E}$;
 \item[(3)] The envelope of holomorphy $E(\partial\Omega)$ of the set $\partial \Omega$ is one-sheeted, and $E(\partial\Omega)=\overline{\Omega}\setminus\mathcal{E}$.
\end{enumerate}
\end{theorem}

The set $\mathcal E$ is obtained as a limit in the Hausdorff metric of a sequence $\{E_\nu\}$ of algebraic hypersurfaces of $\C^n=\C^{n-1}_z\times\C_w$ such that the union of the corresponding sets of ramification points with respect to the
projection $\C^n\to\C^{n-1}_z$ is an everywhere dense subset of $\C^{n-1}_z$. For $n=2$, this idea goes back to Wermer in [W], where an example of a compact set $K$ in $\C^2$ with nontrivial polynomial hull $\widehat{K}$ such that $\widehat{K}\setminus K$ has no analytic structure is given. Wermer's construction was then further exploited and developed in a series of articles [A], [D], [DS], [EM], [Le], [Sl]. Note also that, first, our construction of $\mathcal E$ is slightly different from Wermer's one (the main idea being the same), and, secondly, that, in the general case $n>2$, the situation is substantially more difficult from the technical point of view than that considered by Wermer.

Finally, let us mention a result due to Lupacciolu [Lu] about extendability of $CR$ functions defined on the boundary of an unbounded strictly pseudoconvex domain $\Omega$: Suppose that there exists a divisor which does not meet the domain $\Omega$. Then $E(\partial\Omega)=\overline{\Omega}$; namely, any $CR$ function on the boundary extends inside the domain.\medskip

{\footnotesize\noindent{\bf Acknowledgement.} Part of this work was done while the second author was a visitor at the Scuola Normale Superiore (Pisa) and at the Institut des Hautes \'Etudes Scientifiques (Paris). It is his pleasure to thank these institutions for their hospitality and excellent working conditions. }
\bigskip

%
%
%
%
%
%
\section{Construction of an unbounded Wermer type set in $\C^n$}
\indent Let $(z,w) = (z_1, \ldots, z_{n-1},w)$ denote the coordinates in $\C^n$ and for each $\nu \in \N$ let $\N_\nu \coloneqq \{1, 2, \ldots, \nu\}$. For each $p \in \N_{n-1}$ fix an everywhere dense subset $\{a_l^p\}_{l=1}^\infty$ of $\C$ such that $a_l^p \neq a_{l'}^p$ if $l \neq l'$. Further fix a bijection $\Phi \coloneqq ([\,\cdot\,], \phi) \colon \N \to \N_{n-1} \times \N$ and define a sequence $\{a_l\}_{l=1}^\infty$ in $\C$ by letting $a_l \coloneqq a^{[l]}_{\phi(l)}$. Moreover, let $\{\varepsilon_l\}_{l=1}^\infty$ be a decreasing sequence of positive numbers converging to zero that we consider to be fixed, but that will be further specified later on. Then for every $\nu \in \N$ we define $g_\nu$ to be the algebraic function 
\[ g_\nu(z) \coloneqq \sum_{l=1}^\nu \varepsilon_l \sqrt{z_{[l]} - a_l} \]
and let
\[ E_\nu \coloneqq \big\{ (z,w) \in \C^n : w = g_\nu(z) \big\}. \]
By definition, $g_\nu$ is a multi-valued function that takes $2^\nu$ values at each point $z \in \C^{n-1}$ (counted with multiplicities). Therefore we can always choose single-valued functions $w^{(\nu)}_1, \ldots, w^{(\nu)}_{2^\nu}$ on $\C^{n-1}$ such that 
\[ g_\nu(z) = \big\{w_j^{(\nu)}(z) : j = 1, \ldots, 2^\nu \big\} \]
for all $z \in \C^{n-1}$. Note that these functions are not continuous and that they are not uniquely determined, even though the set $g_\nu(z)$ is well-defined for each $z \in \C^{n-1}$. Indeed, we may freely change the numeration of the values $w^{(\nu)}_1(z), \ldots, w^{(\nu)}_{2^\nu}(z)$ for each $z \in \C^{n-1}$. 

Define for each $\nu\in\N$ a function $P_\nu \colon \C^n \to \C$ as
\[ P_\nu(z,w) \coloneqq \bigr(w-w^{(\nu)}_1(z)\bigl) \cdots \bigr(w - w^{(\nu)}_{2^\nu}(z)\bigl). \]
\begin{lemma} \label{thm_defPn}
The sequence $\{P_\nu\}_{\nu=1}^\infty$ consists of holomorphic polynomials on $\C^n$ and has the following properties:
  \begin{enumerate}
    \item[(1)] $E_\nu = \{(z,w) \in \C^n : P_\nu(z,w) = 0\}$.
    \item[(2)] $P_{\nu+1} \rightarrow P_\nu^2$ uniformly on compact subsets of $\C^n$ as $\varepsilon_{\nu+1} \rightarrow 0$. 
  \end{enumerate}
\end{lemma}
\begin{proof}
First note that if for each $p \in \N_{n-1}$ we let $U_p$ be an open convex subset of $\C$ not meeting $A_\nu^p \coloneqq \{ a_l : l \in \N_\nu, [l] = p \}$, then after possibly renumbering the values $w_j^{(\nu)}(z)$ for $z \in U \coloneqq U_1 \times \cdots \times U_{n-1}$, we can always assume the functions $w^{(\nu)}_1, \ldots, w^{(\nu)}_{2^\nu}$ to be holomorphic on $U$. Since the value $P_\nu(z,w)$ is independent of the numeration of the $w_j^{(\nu)}(z)$, this shows that $P_\nu$ is a holomorphic function outside the set $\mathcal{A_\nu} \coloneqq \{(z,w) \in \C^n : z_p \in A_\nu^p \text{ for some } p \in \N_{n-1}\}$. Observing that $P_\nu$ is locally bounded near each point of $\mathcal{A}_\nu$ and applying Riemann's removable singularities theorem, we conclude that $P_\nu$ is actually holomorphic in the whole of $\C^n$. Then estimating $\abs{P_\nu}$ outside some ball $B^n(0,R) \subset \C^n$ from above by a suitable scalar multiple of $\abs{w^{2^\nu}} + \sum_{p=1}^{n-1} \abs{z_p^{2^{\nu-1}}}$, one can easily see that $P_\nu$ is in fact a holomorphic polynomial. To prove the second part of the lemma we observe that $P_{\nu+1}(z ,w)$ is in fact the product of the $2^\nu$ factors $\big((w-w^{(\nu)}_j(z))^2  - \varepsilon_{\nu+1}^2 (z_{[\nu+1]} - a_{\nu+1}) \big)$, $j \in \N_{2^\nu}$, and hence equals
\[ \sum_{p=0}^{2^\nu} (-1)^p \Big[ \big( \varepsilon_{\nu+1}^{2} ( z_{[\nu+1]}-a_{\nu+1}) \big)^{2^\nu-p} \cdot \!\!\!\!\! \sum_{1 \le j_1 < \cdots < j_p \le 2^\nu} \!\!\!\!\! \big(w-w^{(\nu)}_{j_1}(z)\big)^2 \cdots \big(w-w^{(\nu)}_{j_p}(z)\big)^2 \Big]. \]
Note that for $p = 2^\nu$ the inner sum equals $P_\nu^2(z,w)$. Since $w^{(\nu)}_1, \ldots, w^{(\nu)}_{2^\nu}$ are independent of $\varepsilon_{\nu+1}$ and bounded on compact subsets of $\C^{n-1}$, we conclude that $P_{\nu+1} \rightarrow P_\nu^2$ uniformly on compact subsets as $\varepsilon_{\nu+1} \rightarrow 0$.
\end{proof}

\noindent {\bf Remark.} A more careful consideration shows that one has the following explicit formula for $P_\nu$:
\begin{equation*} \label{equ_Pn}
P_\nu(z,w) = \sum_{d=0}^{2^{\nu-1}} (-1)^d \Big( \sum_{l=1}^\nu \varepsilon_l^2(z_{[l]} - a_l) \Big)^d w^{2^\nu-2d}.
\end{equation*}
\begin{lemma} \label{thm_defE}
Let $\{\varepsilon_l\}$ be chosen in such a way that $\varepsilon_l \sqrt{\abs{z_{[l]} - a_l}} < 1/2^l$ on $B^{n-1}(0,l) \subset \C_z^{n-1}$ for every $l \in \N$. Then the following assertions hold true: 
\begin{enumerate}
  \item[(1)] For every $R > 0$ and $\nu, \mu \in \N$, $\nu \ge R$, the Hausdorff distance between $E_\nu \cap \Bbarn(0,R)$ and $E_{\nu+\mu} \cap \Bbarn(0,R)$ is less than $1/2^\nu$. In particular, the sequence $\{E_\nu \cap \Bbarn(0,R)\}_{\nu=1}^\infty$ converges in the Hausdorff metric to a closed set $\mathcal{E}_{(R)} \subset \Bbarn(0,R)$.
  \item[(2)] The union $\mathcal{E} \coloneqq \bigcup_{R>0} \mathcal{E}_{(R)}$ of all $\mathcal{E}_{(R)}$ is a nonempty closed unbounded subset of $\C^n$ and a point $(z,w) \in \C^n$ lies in $\mathcal{E}$ if and only if there exists a sequence of complex numbers $w_\nu$ converging to $w$ such that $(z,w_\nu) \in E_\nu$ for every $\nu \in \N$. 
  \item[(3)] For each $z \in \C^{n-1}$, the set $\mathcal{E}_z \coloneqq \mathcal{E} \cap \big(\{z\} \times \C\big)$ has zero $2$-dimensional Lebesgue measure.
\end{enumerate}
\end{lemma}
\begin{proof}
Let $\Delta_R \coloneqq \Bbarnn(0,R) \times \C$. For every $\big(z,w_j^{(\nu + \mu)}(z)\big) \in E_{\nu+\mu} \cap \overline{\Delta}_R$ there exists $\big(z,w_k^{(\nu)}\big) \in E_\nu \cap \overline{\Delta}_R$ such that for suitably chosen signs one has 
\[ w_j^{(\nu + \mu)}(z) = w_k^{(\nu)}(z) + \sum_{l=\nu+1}^{\nu+\mu} \pm \varepsilon_l \sqrt{z_{[l]} - a_l} \]
(here, by some abuse of notation, $\sqrt{\cdot}$ denotes a single-valued branch of the multi-valued function $\sqrt{\cdot}$). By assumption we have $\varepsilon_l \abs{\sqrt{z_{[l]}-a_l}\,} = \varepsilon_l \sqrt{\abs{z_{[l]}-a_l}} < 1/2^l$ on $\Bbarnn(0,R)$ for each $l > \nu$. Hence $\abs{w_j^{(\nu + \mu)}(z) - w_k^{(\nu)}(z)} < 1/2^\nu$, and it follows that the Hausdorff distance between $E_{\nu+\mu} \cap \overline{\Delta}_R$ and $E_\nu \cap \overline{\Delta}_R$ is less than $1/2^\nu$. In particular, $\{E_\nu \cap \Bbarn(0,R)\}_{\nu = 1}^\infty$ is a Cauchy sequence in the Hausdorff metric and thus converges to a nonempty closed subset $\mathcal{E}_{(R)} \subset \C^n$. Since $\mathcal{E} \cap \Bbarn(0,R) = \mathcal{E}_{(R)}$ for all $R > 0$, we conclude that $\mathcal{E}$ is closed. Obviously, it is also unbounded and nonempty. The characterization of $(z,w) \in \mathcal{E}$ as a limit of points $(z, w_\nu) \in E_\nu$ follows immediately from the facts that in each bounded neighbourhood of $(z,w)$ the set $\mathcal{E}$ is the limit of $\{E_\nu\}$ in the Hausdorff metric and that $E_\nu \cap \bigl( \{z\} \times \C \bigr) \neq \varnothing$ for all $z \in \C^{n-1}$. Finally, by what we have already proven, we know that the Hausdorff distance between $E_\nu \cap \overline{\Delta}_R$ and $\mathcal{E}_{(R)}$ is not greater than $1/2^\nu$ if $\nu \ge R$. Hence if $z \in \C^{n-1}$ is fixed, the set $\mathcal{E}_z$ is contained in $\{z\} \times \bigcup_{j=1}^{\,2^\nu} \overline{\Delta}^{\begin{minipage}{1ex} $ \scriptstyle \vspace{-1.4ex} 1 $ \end{minipage}}(w_j^{(\nu)}(z), 1/2^\nu)$ for every $\nu \in \N$ large enough (here $\overline{\Delta}^{\begin{minipage}{1ex} $ \scriptstyle \vspace{-1.4ex} 1 $ \end{minipage}}(a,r) \subset \C$ denotes the closed disc centered at the point $a$, of radius $r$). But the volume of the latter set is not greater than $\pi / 2^\nu$; thus $\mathcal{E}_z$ has zero $2$-dimensional Lebesgue measure.
\end{proof}

If $\{\varepsilon_l\}$ converges to zero fast enough, then by the previous lemma the analytic sets $E_\nu$ determine a limit set $\mathcal{E}$. We want to use this set in the construction of our example. To do so, we need to have two specific properties of this set. Namely, we want to ensure that $\mathcal{E}$ has no analytic structure, and we seek a description of $\mathcal{E}$ in terms of certain sublevel sets of the polynomials $P_\nu$. In the next two sections we will show that we indeed can assure $\mathcal{E}$ to have these properties, provided that $\{\varepsilon_l\}$ is converging to zero fast enough.

%
%
%
%
%
%
\section{Choice of the sequence $\{\varepsilon_l\}$ - Part I}
First we want to show that, for $\{\varepsilon_l\}$ decreasing fast enough, the set $\mathcal{E}$ contains no analytic varieties of positive dimension. In order to do so, it obviously suffices to show that $\mathcal{E}$ contains no analytic disc, i.e., there exists no (nonconstant) holomorphic mapping $f \colon \D \to \C^n$ from the unit disc $\D \subset \C$ to $\C^n$ with image completely contained in $\mathcal{E}$. For analytic discs with constant $z$-coordinates this is immediately clear, since we know that $\mathcal{E}_z$ has zero two-dimensional Lebesgue measure for every $z \in \C^{n-1}$. The hard part is to show that there exists no analytic disc $f(\D) \subset \mathcal{E}$ such that the projection $f_z \coloneqq \pi_z \circ f$ onto $\C_z^{n-1}$ is not constant. The general idea is the following: Let $f \colon \D \to \C^n$ be an analytic disc lying in the analytic hypersurface $w = \sqrt{z_p - a}$, $a \in \C$, and such that $f_z \colon \D \to \C_z^{n-1}$ is a biholomorphic embedding of $\D$ into $\C_z^{n-1}$. Then $f_z(\D)$ is either completely contained in the slice $S_a^p \coloneqq \{z \in \C^{n-1} : z_p = a \}$ or does not intersect $S_a^p$ at all. This is due to the fact that if $S_a^p \cap f_z(U) = \{z_0\}$, $U \subset \D$ open and small enough, then for the canonical parametrization $g \colon f_z(U) \to \C_w$ of $f(U)$ and for $\zeta^+, \zeta^- \in \C^{n-1}$ such that $z_0 + \zeta^+, z_0 - \zeta^- \in f_z(U)$, the slope $\abs{g(z_0 + \zeta^+) - g(z_0 - \zeta^-)}/\norm{\zeta^+ + \zeta^-}$ becomes unbounded as $\zeta^+, \zeta^- \to 0$, which contradicts the holomorphicity of $g$. Since each set $E_\nu$ is defined by a sum of terms of the form $\sqrt{z_{[l]} - a_l}$, and since, moreover, the subsequence $\{a_l^p\}_{l=1}^\infty$ of $\{a_l\}$ is dense in $\C$, this will enable us to show that for $\{\varepsilon_l\}$ decreasing fast enough, every analytic disc $f(\D) \subset \mathcal{E}$ must have constant $z_p$-coordinate. Due to the fact that $p \in \N_{n-1}$ here is arbitrary, our assertion will be proved.

There arise some technical difficulties, the most important of which is the following: while for every above-described analytic disc in the analytic hypersurface $w = \sqrt{z_p - a}$ the projection $f_z(\D)$ cannot intersect $S_a^p$ (at least if its $z_p$-coordinate is not already constant), this property might get spoiled when adding further terms $\sqrt{z_{[l]} - a_l}$, $l\in \N$, and thus does not carry over necessarily to the limit set $\mathcal{E}$. In general this problem can be easily handled, except, however, at points $z_0 \in S_a^p$ that are contained in $S^{[l]}_{a_l}$ for more than one $l \in \N$. In this situation there are root branches originating from $z_0$ in different directions $p_1, \ldots, p_T \in \N_{n-1}$, and in general their slopes near the point $z_0$ may cancel out each other. To deal with this problem, we will show that we can at least guarantee the following: for every $z_0 \in S^{[l]}_{a_l} \cap B^{n-1}(0,l)$, $l \in \N$, there does not exist any analytic disc $f(\D) \subset \mathcal{E}$ such that $f_z(\D) \cap S^{[l]}_{a_l} = \{z_0\}$ and such that $f_z(\D)$ is contained in the cone $z_0 + \bigcap_{t=1}^T \Gamma^{p_t}(\alpha)$; here
\[ \Gamma^p(\alpha) \coloneqq \{\zeta \in \C^{n-1} : \zeta_p \neq 0 \;\;\text{and}\;\; \frac{\abs{\zeta_q}}{\abs{\zeta_p}} < \alpha, \;\,\text{for all}\;\, q \in \N_{n-1}, \, q \neq p \}, \]
where $\alpha$ is a positive number that will depend on the choice of $\{\varepsilon_l\}$ (note that if $\zeta \in \Gamma^p(\alpha)$, then also $\lambda \zeta \in \Gamma^p(\alpha)$ for every $\lambda \in \C^\ast$). In fact, the faster $\{\varepsilon_l\}$ decreases, the larger we will be able to choose $\alpha$. It turns out that this weaker assertion is sufficient for our purpose, since locally for every analytic disc $f(\D) \subset \mathcal{E}$ the projection $f_z(\D)$ lies in $\bigcap_{t=1}^T \Gamma^p(\alpha)$ for suitable $p_1, \ldots, p_T \in \N_{n-1}$ and $\alpha > 0$ large enough.

The above complications, as well as most of the other technical difficulties for choosing the sequence $\{\varepsilon_l\}$, do not occur in the case $n = 2$. In fact, in this case the proof becomes relatively simple, and most of the work of this section is not needed. Hence in what follows we will often implicitly assume that $n \ge 3$, though this will not have any influence on the course and correctness of our arguments (for example, the set $\Gamma^p(\alpha) = \C^\ast$ is still well-defined for $n = 2$, though it is obviously not needed in this case).

\noindent {\bf Remark.} Many of the statements in this section involve the function $\sqrt{\cdot} \colon \C \to \C$, which is multivalued. In general, whenever such a statement is made, we will implicitly mean it to hold true for every choice of a single-valued branch $(\sqrt{\cdot}\,)_b \colon \C \to \C$ of $\sqrt{\cdot}$ (no assumptions on continuity). However, there will be cases when we will have to deal with particular single-valued branches of $\sqrt{\cdot}$\,. By some abuse of notation, they will be denoted by the same symbol $\sqrt{\cdot}$\,. We will always point out when $\sqrt{\cdot}$ denotes a particular single-valued branch whenever such a situation first occurs. 

\begin{lemma} \label{thm_defC}
There exists a constant $0<C<1$ such that for all $z, z',\zeta \in \C$,
\[ \sqrt{\abs{\zeta}} \le \big\lvert \sqrt{z + \zeta} - \sqrt{z' - \zeta} \,\big\rvert \le 2\sqrt{\abs{\zeta}} \quad \text{if} \quad \abs{z}, \abs{z'} \le C \abs{\zeta}. \]
\end{lemma}
\begin{proof}
This is immediately clear, since 
\begin{equation*} \frac{\gabs{\sqrt{z + \zeta} - \sqrt{z' - \zeta}\,}}{\sqrt{\abs{\zeta}}} = \Gabs{\sqrt{(z/\zeta) + 1} - \sqrt{(z'/\zeta) - 1}\,} \xrightarrow{z/\zeta,\, z'\!/\zeta \to \, 0} \sqrt{2}\,. \end{equation*}
\end{proof}
\begin{lemma} \label{thm_sqrt}
For every $p \in \N_{n-1}$ and $\alpha > 0$, one has 
\[ \lim_{\zeta \to 0} \frac{\gabs{\sqrt{\zeta_p} - \sqrt{-\zeta_p}}}{2\norm{\zeta}} = + \infty \quad \text{on} \quad \Gamma^p(\alpha). \]
\end{lemma}
\begin{proof}
Indeed, with $c_\alpha \coloneqq \max\{1, \alpha\}$ we have
\[ \frac{\abs{\sqrt{\zeta_p} - \!\sqrt{-\zeta_p}\,}}{2\norm{\zeta}} = \frac{1}{\sqrt{2}} \frac{\sqrt{\abs{\zeta_p}}}{\norm{\zeta}} = \frac{1}{\sqrt{2}} \Big( \sum_{q=1}^{n-1} \frac{\abs{\zeta_q}^2}{\abs{\zeta_p}} \Big)^{-1/2} \ge \frac{1}{\sqrt{2}} \Big( \sum_{q=1}^{n-1} c_\alpha \abs{\zeta_q} \Big)^{-1/2} \]
on $\Gamma^p(\alpha)$, and the last term tends to $+ \infty$ as $\zeta \to 0$.
\end{proof}
\begin{lemma} \label{thm_cones}
Let $P \coloneqq \{p_t\}_{t=1}^T \subset \{1, \ldots, n-1\}$, $p_t\neq p_{t'}$ if $t\neq t'$, and $\{e_t\}_{t=1}^T \subset (0, \infty)$, $T \ge 2$. Define a constant $\alpha > 0$ by $\alpha \coloneqq \min \big\{ \frac{1}{9}(e_m/e_{m+1})^2 : 1 \le m \le T - 1 \big\}$. Then for every $\nu > 0$, there exists a positive number $\delta > 0$ such that
\[ \frac{\gabs{\sum_{m=1}^T e_m \big( \sqrt{z_{p_m} + (\zeta_{p_m} + \zeta_{p_m}')} - \sqrt{z_{p_m} - (\zeta_{p_m} + \zeta_{p_m}'')} \, \big)}}{2\norm{\zeta}} > \nu \]
for every $\zeta \in \big( \bigcap_{m=1}^T \Gamma^{p_m}(\alpha) \big) \cap B^{n-1}(0, \delta)$ and $\zeta', \zeta'', z \in \Delta^{n-1}\big(0, (C/2)\abs{\zeta}_P\big)$. Here $C$ is the constant from Lemma \ref{thm_defC}, $\abs{\zeta}_P \in [0,\infty]^{n-1}$ is defined by $(\abs{\zeta}_P)_p = \abs{\zeta_p}$ if $p \in P$, $(\abs{\zeta}_P)_p = \infty$ if $p \in \complement P \coloneqq \N_{n-1} \setminus P$, and $\Delta^{n-1}\big(0,(r_1, \ldots, r_{n-1})\big) \coloneqq \{z \in \C^{n-1} : \abs{z_p} < r_p, \text{ if } r_p > 0, \text{ or } z_p = 0, \text{ if } r_p = 0, \;p \in \N_{n-1}\}$ for $r \in [0, \infty]^{n-1}$.
\end{lemma}
\noindent {\bf Remark.} The statement of this lemma is interesting and will be used only in the case when $\alpha > 1$ (otherwise the intersection $\bigcap_m \Gamma^{p_m}(\alpha)$ is empty).
\begin{proof}
For every $m \in \N_{T-1}$ we define $\alpha_m \coloneqq \frac{1}{9} (e_m/e_{m+1})^2$, and for every $m \in \N_T$ we let $D_m(\zeta) \coloneqq \{ z \in \C^{n-1} : \abs{z_{p_m}} \le C\abs{\zeta_{p_m}}\}$. We will show by induction that for every $t = 1, \ldots, T$, the inequality
\begin{equation} \label{equ_Ht}
\Gabs{ \sum_{m=1}^t e_m \big( \sqrt{z_{p_m}' + \zeta_{p_m}} - \sqrt{z_{p_m}'' - \zeta_{p_m}}\, \big) } \ge e_t \textstyle \sqrt{\abs{\zeta_{p_t}}}
\end{equation}
holds true for $\zeta \in \bigcap_{m = 1}^{t-1} \Gamma^{p_m}(\alpha_m)$ and $z', z'' \in \bigcap_{m = 1}^{t} D_m(\zeta)$. Indeed, the case $t = 1$ is already proven by Lemma \ref{thm_defC}. For the step $t \to t+1$, let $H_{t+1}$ denote the left term in $(\ref{equ_Ht})$ where the sum is taken up to $t+1$. Using the induction hypothesis and applying Lemma \ref{thm_defC}, we see that
\begin{equation*} \begin{split}
  H_{t+1} & \ge \!\Gabs{\! \sum_{m=1}^t \! e_m \textstyle \big( \sqrt{z_{p_m}' \!\!+ \zeta_{p_m}} \!-\! \sqrt{z''_{p_m} \!\!- \zeta_{p_m}} \, \big) } - e_{t+1} \Gabs{\sqrt{z_{p_{t+1}}' \!\!+ \zeta_{p_{t+1}}} \!-\! \sqrt{z_{p_{t+1}}'' \!\!- \zeta_{p_{t+1}}}\,} \\
  & \ge e_t \textstyle \sqrt{\abs{\zeta_{p_t}}} - 2 e_{t+1} \sqrt{\abs{\zeta_{p_{t+1}}}} \displaystyle \qquad \text{for} \;\; \zeta \in \!\!\bigcap_{m=1}^{t-1} \Gamma^{p_m}(\alpha_m),\;\, z', z'' \in \!\!\bigcap_{m = 1}^{t + 1} D_m(\zeta).
\end{split} \end{equation*}
Observe that there is nothing to show in the case $\zeta_{p_{t+1}} = 0$. Hence we can assume $\zeta_{p_{t+1}} \neq 0$ and write 
\[ e_t \textstyle \sqrt{\abs{\zeta_{p_t}}} - 2 e_{t+1} \textstyle \sqrt{\abs{\zeta_{p_{t+1}}}} = 2e_{t+1} \sqrt{\abs{\zeta_{p_{t+1}}}} \displaystyle \Big( \frac{e_t}{2e_{t+1}} \frac{\sqrt{\abs{\zeta_{p_t}}}}{\sqrt{\abs{\zeta_{p_{t+1}}}}} - 1 \Big). \]
One immediately checks that the term between the brackets is not less than $1/2$ precisely if $\abs{\zeta_{p_{t+1}}} / \abs{\zeta_{p_t}} \le \alpha_t$; hence
\[ e_t \textstyle \sqrt{\abs{\zeta_{p_t}}} - 2e_{t+1} \sqrt{\abs{\zeta_{p_{t+1}}}} \displaystyle \ge e_{t+1} \textstyle \sqrt{\abs{\zeta_{p_{t+1}}}} \qquad \text{for} \;\; \zeta \in \Gamma^{p_t}(\alpha_t). \]
This completes the induction and proves $(\ref{equ_Ht})$. But from Lemma \ref{thm_sqrt} we know that 
\[ \lim_{\zeta \to 0} \frac{\abs{\sqrt{\zeta_{p_T}} - \sqrt{-\zeta_{p_T}}}}{2\norm{\zeta}} = +\infty \qquad \text{on} \;\; \Gamma^{p_T}(\alpha_T), \]
where $\alpha_T \coloneqq \alpha$. Combining this with the estimate $(\ref{equ_Ht})$ in the case $t=T$, we conclude that for every $\nu > 0$ there exists $\delta > 0$ such that
\[ \frac{\gabs{\sum_{m=1}^T e_m \big( \sqrt{z_{p_m}' + \zeta_{p_m}} - \sqrt{z_{p_m}'' - \zeta_{p_m}} \big)}}{2\norm{\zeta}} > \nu \]
for $\zeta \in \bigcap_{m = 1}^{T} \Gamma^{p_m}(\alpha_m) \cap B^{n-1}(0, \delta)$ and $z', z'' \in \bigcap_{m = 1}^{T} D_m(\zeta) = \Delta^{n-1}(0, C\abs{\zeta}_P)$. Since $\alpha \le \alpha_m$ for all $m \in \N_T$ and $\Gamma^p(\alpha) \subset \Gamma^p(\alpha')$ for $\alpha \le \alpha'$, this concludes the proof. Indeed, for $\zeta', \zeta'', z \in \Delta^{n-1}\big(0, (C/2)\abs{\zeta}_P\big)$, the points $z' \coloneqq z + \zeta'$ and $z'' \coloneqq z - \zeta''$ always satisfy $z', z'' \in \Delta^{n-1}(0, C\abs{\zeta}_P)$. 
\end{proof}

We want to estimate the slope between two points of the set $E_\nu$ when their projection to $\C^{n-1}_z$ lies near the zero set of one of the functions $\sqrt{z_{[l]} - a_l}$, $l=1, \ldots, \nu$. For this we need some notations: For every $\nu \in \N$ and $p \in \N_{n-1}$ we define
\[ S_\nu \coloneqq \{ \zeta \in \C^{n-1} : \zeta_{[\nu]} = a_\nu \}, \qquad S^p \coloneqq \{ \zeta \in \C^{n-1} : \zeta_p = 0 \}, \]
and
\[ L_\nu^p \coloneqq \{l \in \N : 1 \le l \le \nu, \, [l] = p \}, \qquad A_\nu^p \coloneqq \{ a_l \in \C : l \in L_\nu^p \}. \]
Obviously $\bigcup_{p=1}^{n-1} L_\nu^p = \N_\nu$. Moreover, if $z \in \C^{n-1}$, we define
\[ L_\nu^p(z) \coloneqq \{ l \in L_\nu^p : z_p = a_l \}. \]
Note that $L_\nu^p(z)$ consists of at most one element. Further, for $P \subset \N_{n-1}$ such that $[\nu] \in P$ and $z \in S_\nu$ we let
\[ \sL_\nu^P(z) \coloneqq \bigcup_{p \in P} L_\nu^p(z). \]
Observe that under the assumptions on $P$ and $z$, we always have $\nu \in \sL_\nu^P(z)$. As mentioned before, the case $\abs{\sL_\nu^P(z)} > 1$ is of special interest and leads us to consider the sets $\bigcap_p \Gamma^p(\alpha)$ for $\alpha > 1$. Here $\alpha$ was claimed to depend on $\{\varepsilon_l\}$, and we now clarify this dependence by the following definition: for every $\nu \in \N$, $P \subset \N_{n-1}$ such that $[\nu] \in P$ and every $z \in S_\nu$, let $\alpha_\nu^P(z)$ be the positive number
\[ \alpha_\nu^P(z) \coloneqq \left\{ \renewcommand{\arraystretch}{1.25} \begin{array}{c@{\quad \text{if }\,}l} \nu + 1 & \sL_\nu^P(z) = \{\nu\} \\ \min \big\{\textstyle \frac{1}{9} \displaystyle (\varepsilon_l / \varepsilon_{l'})^2 : l,l' \in \sL_\nu^P(z), l' > l \big\} & \sL_\nu^P(z) \supsetneq \{\nu\}. \end{array} \renewcommand{\arraystretch}{1} \right. \]
Observe that, since the sequence $\{\varepsilon_l\}$ is still in our hands, we can always assume that $\alpha_\nu^P(z) > 1$. Finally, for each $P \subset \N_{n-1}$ and $\alpha > 0$ we let
\[ \gamma(P,\alpha) \coloneqq \Big( \bigcap_{p \in P} \Gamma^p(\alpha) \Big) \cap \Big( \bigcap_{p \in \complement P} S^p \Big). \]
\begin{lemma} \label{thm_wjwk}
Suppose $\varepsilon_1, \ldots, \varepsilon_\nu$ have already been chosen. Let $\delta > 0$. Then for every $z_0 \in S_\nu$ and $P \subset \N_{n-1}$ such that $[\nu] \in P$, there exist $r^P(z_0) > 0$ and $\delta^P(z_0) \in (0, \delta)$ such that for every $j,k \in \N_{2^\nu}$ the inequality
\begin{equation} \label{equ_wjwk} \frac{\gabs{ w_j^{(\nu)}\big(z + (\zeta + \zeta')\big) - w_k^{(\nu)}\big(z - (\zeta + \zeta'') \big) }}{2\norm{\zeta}} > \nu \end{equation}
holds for every $z \in B^{n-1}\big(z_0, r^P(z_0)\big)$, $\zeta \in \gamma\big(P, \alpha_\nu^P(z_0)\big) \cap \partial B^{n-1}\big(0,\delta^P(z_0)\big)$ and $\zeta', \zeta'' \in \Delta^{n-1}\big(0,(C/2)\abs{\zeta}\big)$; here $\abs{\zeta} = (\abs{\zeta_1}, \ldots, \abs{\zeta_{n-1}})$.
\end{lemma}
\begin{proof}
Fix $z_0 \in S_\nu$ and $P \subset \N_{n-1}$ such that $[\nu] \in P$. For each $p \in \N_{n-1}$, let $U_p \subset \C$ be an open convex neighbourhood of $z_{0,p}$ such that 
\[ U_p \cap A_\nu^p = \left\{ \begin{array}{c@{\quad \text{if }\,}l} \varnothing & L_\nu^p(z_0) = \varnothing \\ \{z_{0,p}\} & L_\nu^p(z_0) \neq \varnothing \end{array} \right. \]
and let $U \coloneqq U_1 \times \cdots \times U_{n-1}$. Choose $r > 0$ so small that $B^{n-1}(z_0,2r) \subset U$. For each $l \in \N_\nu$, consider a single-valued branch of the multi-valued function $\sqrt{z_{[l]} - a_l}$ which will also be denoted here by $\sqrt{z_{[l]} - a_l}$. Since for every $l \in \N_\nu \setminus \bigcup_{p=1}^{n-1} L_\nu^p(z_0)$ we have $a_l \notin U_{[l]}$, we can assume that $\sqrt{z_{[l]} - a_l}$ is holomorphic on $U$ for these $l$. After possibly changing the numeration of the roots of $P_\nu(z, \,\cdot\,)$ for $z \in U$, we may further assume for every $h \in \N_{2^\nu}$ that $w_h^{(\nu)}(z) = \sum_{l=1}^\nu \pm \varepsilon_l \sqrt{z_{[l]} - a_l}$ on $B^{n-1}(z_0, 2r)$ for suitably chosen signs depending only on $l$ and $h$. Now define $\tilde{w}_h \colon B^{n-1}(z_0, 2r) \to \C$ as
\begin{equation} \label{equ_wh}
\tilde{w}_h(z) \coloneqq \sum_{p \in P} \; \sum_{l \in L_\nu^p \setminus L_\nu^p(z_0)} \!\! \pm \varepsilon_l \sqrt{z_{[l]} - a_l} \,\, + \sum_{p \in \complement P} \sum_{l \in L_\nu^p} \pm \varepsilon_l \sqrt{z_{[l]} - a_l}\,.
\end{equation}
Since $\N_\nu = \bigcup_{p=1}^{n-1} L_\nu^p$, we obviously have $w_h^{(\nu)}(z) = \tilde{w}_h(z) + \sum_{l \in \sL_\nu^P(z_0)} \pm \varepsilon_l \sqrt{z_{[l]} - a_l}$ on $B^{n-1}(z_0,2r)$. Let $\N_{2^\nu}^2 \coloneqq \N_{2^\nu} \times \N_{2^\nu}$ and $\N_{2^\nu}^2(z_0) \coloneqq \{(j,k) \in \N_{2^\nu}^2 : \tilde{w}_j(z_0) = \tilde{w}_k(z_0) \}$.

\noindent \textsc{Step 1}: We show that there exist $r' > 0$ and $\delta' \in (0,\delta)$ such that $(\ref{equ_wjwk})$ holds for every $\zeta \in B^{n-1}(0,\delta')$, $\zeta', \zeta'' \in \Delta^{n-1}\big(0,(C/2)\abs{\zeta}\big)$, $z \in B^{n-1}(z_0, r')$ and $(j,k) \in \N_{2^\nu}^2 \setminus \N_{2^\nu}^2(z_0)$.

For $l \in L_\nu^p(z_0)$, we have $z_{0,[l]} = a_l$ and $\sqrt{\cdot}$ is continuous at the origin; hence we conclude from $(\ref{equ_wh})$ and the holomorphicity of $\sqrt{z_{[l]} - a_l}$ for $l \in L_\nu^p \setminus L_\nu^p(z_0)$ that $\tilde{w}_h$ is continuous at $z_0$ for every $h \in \N_{2^\nu}$. Thus there exist $M > 0$ and $r_1 > 0$ such that $\abs{\tilde{w}_j(z + (\zeta + \zeta')) - \tilde{w}_k(z - (\zeta + \zeta''))} > M$ for every $\zeta \in B^{n-1}(0,r_1)$, $\zeta', \zeta'' \in \Delta^{n-1}\big(0,(C/2)\abs{\zeta}\big)$, $z \in B^{n-1}(z_0,r_1)$ and $(j,k) \in \N_{2^\nu}^2 \setminus \N_{2^\nu}^2(z_0)$. Moreover, since again $z_{0,[l]} = a_l$ for $l \in \sL_\nu^P(z_0)$ and $\sqrt{\cdot}$ is continuous at the origin, there exists $r_2 > 0$ such that $\sqrt{\abs{(z_{[l]} \pm (\zeta_{[l]} + \tilde{\zeta}_{[l]})) - a_l}} < M / \big(4(n-1)\varepsilon_l\big)$, where $\tilde{\zeta} \in \{\zeta', \zeta''\}$, for every $\zeta \in B^{n-1}(0,r_2)$, $\zeta', \zeta'' \in \Delta^{n-1}\big(0,(C/2)\abs{\zeta}\big)$, $z \in B^{n-1}(z_0,r_2)$ and $l \in \sL_\nu^P(z_0)$. Let $r' \coloneqq \min \{r, r_1, r_2\}$ and $\delta' \coloneqq \min \{\delta, r, r_1, r_2, M/4\nu\}$. Then the following estimate holds true for every $\zeta \in B^{n-1}(0, \delta')$, $\zeta', \zeta'' \in \Delta^{n-1}\big(0,(C/2)\abs{\zeta}\big)$, $z \in B^{n-1}(z_0, r')$ and $(j,k) \in \N_{2^\nu}^2 \setminus \N_{2^\nu}^2(z_0)$:
\begin{eqnarray*}
      \lefteqn{ \frac{\gabs{w_j^{(\nu)}\big(z + \!(\zeta + \zeta')\big) - w_k^{(\nu)}\big(z - (\zeta + \zeta'')\big)}}{2\norm{\zeta}} \ge \frac{\gabs{\tilde{w}_j\big(z + \! (\zeta + \zeta')\big) - \tilde{w}_k\big(z - (\zeta + \zeta'')\big)}}{2\norm{\zeta}} \,- }\\
&& - \frac{ \sum_{l \in \sL_\nu^P(z_0)} \varepsilon_l \big( \sqrt{\gabs{\big(z_{[l]} + (\zeta_{[l]} + \zeta_{[l]}')\big) - a_l}} + \sqrt{\gabs{\big(z_{[l]} - (\zeta_{[l]} + \zeta_{[l]}'') \big) - a_l}} \, \big)}{2\norm{\zeta}} \hspace{8ex} \\ 
&& > \frac{M - \sum_{l \in \sL_\nu^P(z_0)} 2 \varepsilon_l M/\big(4(n-1)\varepsilon_l\big)}{2\norm{\zeta}} \ge \frac{M - M/2}{2\norm{\zeta}} > \nu.
\end{eqnarray*}

\noindent \textsc{Step 2}: We show that there exist $r'' \in (0, r')$ and $\delta'' \in (0, \delta')$ such that $(\ref{equ_wjwk})$ holds for every $\zeta \in \gamma\big(P,\alpha_\nu^P(z_0)\big) \cap K^{n-1}(\delta''/2, \delta'')$, $\zeta', \zeta'' \in \Delta^{n-1}\big(0, (C/2) \abs{\zeta}\big)$, $z \in \Delta^{n-1}\big(z_0, (C/2)\abs{\zeta}_P\big) \cap B^{n-1}(z_0, r'')$ and $(j,k) \in \N_{2^\nu}^2(z_0)$, where for $R_1, R_2 \ge 0$ we put $K^{n-1}(R_1, R_2) \coloneqq \{z \in \C^{n-1} : R_1 < \norm{z} < R_2 \}$.

Observe that the first term in $(\ref{equ_wh})$ is holomorphic in $B^{n-1}(z_0,2r)$ and the second term is constant on the set $z_0 + \bigcap_{p \in \complement P} S^p$. Therefore we can find $M > 0$ and $\tilde{r} > 0$ such that
\[ \frac{\gabs{\tilde{w}_j\big(z_0 + (\zeta + \zeta')\big) - \tilde{w}_k\big(z_0 - (\zeta + \zeta'')\big)}}{2\norm{\zeta}} < M \]
for all $\zeta \in \big( \bigcap\nolimits_{p \in \complement P} S^p \big) \cap B^{n-1}(0, \tilde{r})$, $\zeta', \zeta'' \in \Delta^{n-1}\big(0, (C/2) \abs{\zeta}\big)$ and $(j,k) \in \N_{2^\nu}^2(z_0)$. Moreover, since $z_{0,[l]} = a_l$ for every $l \in \sL_\nu^P(z_0)$, we have $\sqrt{\big(z_{[l]} \pm (\zeta_{[l]} + \tilde{\zeta}_{[l]})\big) - a_l} = \sqrt{(z_{[l]} - z_{0,[l]}) \pm (\zeta_{[l]} + \tilde{\zeta}_{[l]})}$, where $\tilde{\zeta} \in \{\zeta', \zeta''\}$. Hence, using Lemma \ref{thm_defC} and \ref{thm_sqrt} if $\sL_\nu^P(z_0) = \{\nu\}$ and Lemma \ref{thm_cones} if $\sL_\nu^P(z_0) \supsetneq \{\nu\}$, there exists $\tilde{\delta} > 0$ such that 
\[ \frac{\gabs{\sum_{l \in \sL_\nu^P(z_0)} \varepsilon_l \Big( \sqrt{\big(z_{[l]} + (\zeta_{[l]} + \zeta_{[l]}')\big) - a_l} - \sqrt{\big(z_{[l]} - (\zeta_{[l]} + \zeta_{[l]}'')\big) - a_l}\, \Big)}}{2\norm{\zeta}} > \nu + M \]
for all $\zeta \in \big[ \bigcap_{p \in P}\Gamma^{p}\big(\alpha_\nu^P(z_0)\big) \big] \cap B^{n-1}(0, \tilde{\delta})$, $\zeta', \zeta'' \in \Delta^{n-1}\big(0, (C/2)\abs{\zeta}_P\big)$ and $z \in \Delta^{n-1}\big(z_0, (C/2)\abs{\zeta}_P\big)$ (recall the definition of $\alpha_\nu^P(z_0)$). Now choose $\delta''$ such that $0 < \delta'' < \min\{\tilde{r}, \tilde{\delta}, \delta' \}$. Observe that $\tilde{w}_h$ is continuous in $z_0 + \big[ \big( \bigcap_{p \in \complement P} S^p \big) \cap B^{n-1}(0, 2r) \big]$ for every $h \in \N_{2^\nu}$. Hence there exists $r'' \in (0, r')$ such that the following estimate holds true for every $\zeta \in \big( \bigcap\nolimits_{p \in \complement P} S^p \big) \cap K^{n-1}(\delta''/2, \delta'')$, $\zeta', \zeta'' \in \Delta^{n-1}\big(0, (C/2) \abs{\zeta}\big)$, $z \in B^{n-1}(z_0, r'')$ and $(j,k) \in \N_{2^\nu}^2(z_0)$:
\[ \frac{\gabs{\tilde{w}_j\big(z + (\zeta + \zeta')\big) - \tilde{w}_k\big(z - (\zeta + \zeta'')\big)}}{2\norm{\zeta}} < M. \]
Thus for every $\zeta \in \gamma\big(P,\alpha_\nu^P(z_0)\big) \cap K^{n-1}(\delta''/2, \delta'')$, $\zeta', \zeta'' \in \Delta^{n-1}\big(0, (C/2) \abs{\zeta}\big)$, $z \in \Delta^{n-1}\big(z_0, (C/2)\abs{\zeta}_P\big) \cap B^{n-1}(z_0, r'')$ and $(j,k) \in \N_{2^\nu}^2(z_0)$ we get
\begin{equation*} \begin{split}
 & \frac{\abs{w_j^{(\nu)}\big(z \!+\! (\zeta + \zeta')\big) - w_k^{(\nu)}\big(z \!-\! (\zeta + \zeta'')\big)}}{2\norm{\zeta}} \\
 & \ge \frac{\Gabs{\sum_{l \in \sL_\nu^P(z_0)} \varepsilon_l \Big( \sqrt{\big(z_{[l]} + (\zeta_{[l]} + \zeta_{[l]}')\big) - a_l} - \sqrt{\big(z_{[l]} - (\zeta_{[l]} + \zeta_{[l]}'')\big) - a_l}\, \Big)}}{2\norm{\zeta}} - \\
 & - \frac{\gabs{\tilde{w}_j\big(z + (\zeta + \zeta')\big) - \tilde{w}_k\big(z - (\zeta + \zeta'')\big)}}{2\norm{\zeta}} > \nu.
\end{split} \end{equation*}

\noindent \textsc{Step 3}: We show that there exist $r^P(z_0) > 0$ and $\delta^P(z_0) \in (0,
\delta)$ such that $(\ref{equ_wjwk})$ holds for every $\zeta \in \gamma\big(P, \alpha_\nu^P(z_0)\big) \cap \partial B^{n-1}\big(0,\delta^P(z_0)\big)$, $\zeta', \zeta'' \in \Delta^{n-1}\big(0,(C/2)\abs{\zeta}\big)$, $z \in B^{n-1}\big(z_0, r^P(z_0)\big)$ and $j,k \in \N_{2^\nu}$.

We already know that $(\ref{equ_wjwk})$ holds for every $\zeta \in \gamma\big(P,\alpha_\nu^P(z_0)\big) \cap K^{n-1}(\delta''/2, \delta'')$, $\zeta', \zeta'' \in \Delta^{n-1}\big(0, (C/2) \abs{\zeta}\big)$, $z \in \Delta^{n-1}\big(z_0, (C/2)\abs{\zeta}_P\big) \cap B^{n-1}(z_0, r'')$ and $j,k \in \N_{2^\nu}$. It only remains to make proper choices for the constants $r^P(z_0)$ and $\delta^P(z_0)$. First, choose any $\delta^P(z_0)$ such that $\delta'' > \delta^P(z_0) > \delta''/2$. Then there exists $K > 0$ such that
\[ \abs{\zeta_p} > K \quad \text{for all} \quad \zeta \in \gamma\big(P, \alpha_\nu^P(z_0)\big) \cap \partial B^{n-1}\big(0, \delta^P(z_0)\big), \, p \in P. \]
Indeed, let $p \in P$. Then for $\zeta \in \gamma\big(P, \alpha_\nu^P(z_0)\big)$ we have in particular $\zeta \in \Gamma^p\big(\alpha_\nu^P(z_0)\big)$ and hence $\abs{\zeta_q} / \abs{\zeta_p} < \alpha_\nu^P(z_0)$ for every $q \in  \N_{n-1}$ (assuming that $\alpha_\nu^P(z_0) > 1$, which is the only interesting case). Thus $\norm{\zeta} < \alpha_\nu^P(z_0) \sqrt{n-1} \, \abs{\zeta_p}$. Since also $\zeta \in  \partial B^{n-1}\big(0, \delta^P(z_0)\big)$, we conclude that $\abs{\zeta_p} > \delta^P(z_0) / \big(\alpha_\nu^P(z_0) \sqrt{n-1}\,\big) \eqqcolon K$. Now choose $\rho > 0$ such that $\abs{z_p - z_{0,p}} < (CK)/2$ for all $z \in B^{n-1}(z_0, \rho)$ and $p \in P$, i.e., $B^{n-1}(z_0, \rho) \subset \Delta^{n-1}\big(z_0, (C/2)\abs{\zeta}_P\big)$ for all $\zeta \in \gamma\big(P, \alpha_\nu^P(z_0)\big) \cap \partial B^{n-1}\big(0, \delta^P(z_0)\big)$. Then $r^P(z_0) \coloneqq \min\{r'', \rho\}$ is a desired constant. 
\end{proof}

Fix $\nu \in \N$. By the previous lemma we have assigned positive numbers $r^P(z_0)$, $\delta^P(z_0)$ to every $z_0 \in S_\nu$. As we shall see in the proof of Lemma \ref{thm_choiceeps}, the choice of $\varepsilon_{\nu+1}$ will depend on the numbers $\delta^P(z_0)$, $z_0 \in S_\nu$; in fact, we will need a positive lower bound for the set $\{\delta^P(z_0) : z_0 \in S_\nu\}$. However, such a bound does not always exist. Hence from now on we restrict our attention to the compact subset $S_\nu \cap \Bbarnn(0,\nu)$ of $S_\nu$. This set can be covered by finitely many balls $B^{n-1}\big(z_0, r^P(z_0)\big)$, $z_0 \in S_\nu$, and thus leads to a finite set $\{\delta^P(z_1), \ldots, \delta^P(z_m)\} \subset (0, \infty)$ (which of course has a positive minimum). On the way, we have to choose the numbers $r^P(z_0)$ in the covering $\big\{B^{n-1}\big(z_0,r^P(z_0)\big)\big\}_{z_0 \in S_\nu}$ small enough in order to limit the influence of points $z_0 \in S_\nu$ with small value $\alpha^P_\nu(z_0)$. For this purpose, we need some further notations: Fix a decreasing sequence $\{\rho_\nu\}$ of positive numbers converging to zero, such that
\[ \max_{1 \le p \le n-1} \vol \Big( \bigcup_{l \in L_\nu^p} \Delta^1(a_l,\rho_\nu) \Big) \to 0 \quad \text{for} \quad \nu \to \infty. \] 
Then for every $\nu \in \N$, $p \in \N_{n-1}$ and $z \in \C^{n-1}$ we let
\[ \tilde{L}_\nu^p(z) \coloneqq \{ l \in L_\nu^p : \abs{z_p - a_l} \le \rho_\nu \}. \]
Moreover if $z \in S_\nu$ and $P \subset \N_{n-1}$ such that $[\nu] \in P$  we let
\[ \tilde{\sL}_\nu^P(z) \coloneqq \bigcup_{p \in P} \tilde{L}_\nu^p(z). \]
Note that under the assumptions on $P$ and $z$ we always have $\nu \in \tilde{\sL}_\nu^P(z)$. Hence
\[ \tilde{\alpha}_\nu^P(z) \coloneqq \left\{ \renewcommand{\arraystretch}{1.25} \begin{array}{c@{\quad \text{if }\,}l} \nu + 1 & \tilde{\sL}_\nu^P(z) = \{\nu\} \\ \min \big\{ \nu +1, \min \{\textstyle \frac{1}{9} \displaystyle (\varepsilon_l / \varepsilon_{l'})^2 : l,l' \in \tilde{\sL}_\nu^P(z), l' > l \} \big\} & \tilde{\sL}_\nu^P(z) \supsetneq \{\nu\} \end{array} \renewcommand{\arraystretch}{1} \right. \]
is a well-defined positive number.

\begin{corollary} \label{thm_mainest}
Suppose $\varepsilon_1, \ldots, \varepsilon_\nu$ have already been chosen. Let $\delta > 0$. Then there exists a finite subset $D_\nu \coloneqq \{\delta_\nu^1, \ldots, \delta_\nu^{d_\nu}\} \subset (0,\delta)$ such that for every $z \in S_\nu \cap B^{n-1}(0,\nu)$ and $P \subset \N_{n-1}$ such that $[\nu] \in P$, there exists some $\sigma \in \{1, \ldots, d_\nu \}$ such that for every $j,k \in \N_{2^\nu}$ the inequality
\begin{equation} \label{equ_mainest}
\frac{\gabs{ w_j^{(\nu)}\big(z + (\zeta + \zeta')\big) - w_k^{(\nu)}\big(z - (\zeta + \zeta'') \big) }}{2\norm{\zeta}} > \nu
\end{equation}
holds true for all $\zeta \in \gamma\big(P, \tilde{\alpha}_\nu^P(z) \big) \cap \partial B^{n-1}(0, \delta_\nu^\sigma)$ and $\zeta', \zeta'' \in \Delta^{n-1}\big(0, (C/2)\abs{\zeta}\big)$.
\end{corollary}
\begin{proof}
By the previous lemma, for every $z_0 \in S_\nu$ and $P \subset \N_{n-1}$, $[\nu] \in P$, there exist positive numbers $r^P(z_0) \in (0,\rho_\nu)$ and $\delta^P(z_0) \in (0,\delta)$ such that $(\ref{equ_mainest})$ holds for every $j,k \in \N_{2^\nu}$, $z \in B^{n-1}\big(z_0, r^P(z_0)\big)$, $\zeta \in \gamma\big(P, \alpha_\nu^P(z_0)\big) \cap \partial B^{n-1}\big(0,\delta^P(z_0)\big)$ and $\zeta', \zeta'' \in \Delta^{n-1}\big(0,(C/2)\abs{\zeta}\big)$. Let 
\[ r(z_0) \coloneqq \min \big\{ r^P(z_0) : P \subset \N_{n-1} \text{ such that } [\nu] \in P \big\}. \]
By compactness of $S_\nu \cap \Bbarnn(0, \nu)$, there exist finitely many points $z_1, \ldots, z_M \in S_\nu$ such that $S_\nu \cap \Bbarnn(0, \nu) \subset \bigcup_{m=1}^M B^{n-1}\big(z_m, r(z_m)\big)$. Let
\[ D_\nu \coloneqq \big\{ \delta^P(z_m) : P \subset \N_{n-1} \text{ such that } [\nu] \in P, \, m = 1, \ldots, M \big\}. \]
Then for every $z \in S_\nu \cap B^{n-1}(0,\nu)$ and $P \subset \N_{n-1}$, $[\nu] \in P$, there exist $\sigma \in \{1, \ldots, d_\nu \}$ and $m \in \N_M$ such that $\abs{z - z_m} \le \rho_\nu$ and such that $(\ref{equ_mainest})$ holds for every $j,k \in \N_{2^\nu}$, $\zeta \in \gamma\big(P, \alpha_\nu^P(z_m) \big) \cap \partial B^{n-1}(0, \delta_\nu^\sigma)$ and $\zeta', \zeta'' \in \Delta^{n-1}\big(0, (C/2)\abs{\zeta}\big)$. It remains to observe that we herein can replace $\alpha_\nu^P(z_m)$ by $\tilde{\alpha}_\nu^P(z)$. Indeed, since $\abs{z - z_m} \le \rho_\nu$, we have $L_\nu^p(z_m) \subset \tilde{L}_\nu^p(z)$ for all $p \in \N_{n-1}$ and thus $\sL_\nu^P(z_m) \subset \tilde{\sL}_\nu^P(z)$. Recalling the definitions of $\alpha_\nu^P(z_m)$ and $\tilde{\alpha}_\nu^P(z)$, we conclude that $\tilde{\alpha}_\nu^P(z) \le \alpha_\nu^P(z_m)$. In particular, we get $\gamma \big( P, \tilde{\alpha}_\nu^P(z) \big) \subset \gamma \big( P, \alpha_\nu^P(z_m) \big)$.
\end{proof}

We are now able to specify the choice of the sequence $\{\varepsilon_l\}$:

\begin{lemma} \label{thm_choiceeps}
If $\{\varepsilon_l\}$ is decreasing fast enough, then for every fixed $\nu \in \N$ and for every $z \in S_\nu \cap B^{n-1}(0,\nu)$ and $P \subset \N_{n-1}$ such that $[\nu] \in P$, there exists $\delta \in (0, 1/\nu)$ such that
\begin{equation} \label{equ_slopeE}
  \frac{w' - w''}{\norm{\zeta' + 2\zeta + \zeta''}} \ge \frac{\nu - 1}{1 + (C/2)} \quad \text{for all} \quad w' \in \mathcal{E}_{z + (\zeta + \zeta')}, \, w'' \in \mathcal{E}_{z - (\zeta + \zeta'')}
\end{equation}
and all choices of $\zeta \in \gamma\big(P, \tilde{\alpha}_\nu^P(z) \big) \cap \partial B^{n-1}(0,\delta)$ and $\zeta', \zeta'' \in \Delta^{n-1}\big(0, (C/2)\abs{\zeta} \big)$. Moreover, $\frac{1}{9}(\varepsilon_l / \varepsilon_{l+1})^2 > l$ and $\varepsilon_l \sqrt{\abs{z_{[l]} - a_{l}}} <  \frac{1}{2^l}$ on $B^{n-1}(0,l)$.
\end{lemma}
\begin{proof}
We proceed by induction on $l$ and simultaneously choose a sequence $(D_l)$ of finite subsets $D_l = \{\delta_l^1, \ldots, \delta_l^{d_l} \} \subset (0,1/l)$ such that $\varepsilon_l \sqrt{\abs{z_{[l]} - a_{l}}} <  \frac{1}{2^l} \min \{ \delta' \in D_\nu : 1 \le \nu \le l-1\}$ for every $z \in B^{n-1}(0, l+1)$. First let $\varepsilon_1 \coloneqq 1$ and let $D_1 \subset (0,1)$ be the set provided by Corollary \ref{thm_mainest} in the case $\nu , \delta = 1$. If $\varepsilon_1, \ldots, \varepsilon_l$ and $D_1, \ldots, D_l$ have already been chosen, we choose $\varepsilon_{l+1} > 0$ so small that $\frac{1}{9}(\varepsilon_l / \varepsilon_{l+1})^2 > l$ and $\varepsilon_{l+1} \sqrt{\abs{z_{[l+1]} - a_{l+1}}} <  \frac{1}{2^{l+1}} \min \{ \delta' \in D_\nu : 1 \le \nu \le l\}$ for $z \in B^{n-1}(0, l+2)$. Observe that every $\varepsilon_{l+1}' \in (0, \varepsilon_{l+1})$ would also be a proper choice for $\varepsilon_{l+1}$. We then take for $D_{l+1}$ the set provided by Corollary \ref{thm_mainest} in the case $\nu = l+1$ and $\delta = 1/(l+1)$.

Fix $\nu \in \N$, $z \in S_\nu \cap B^{n-1}(0,\nu)$ and $P \subset \N_{n-1}$ such that $[\nu] \in P$. Then by choice of $D_\nu$, there exists $\delta \in D_\nu$ such that estimate $(\ref{equ_mainest})$ holds true for all $j,k \in \N_{2^\nu}$ and all considered $\zeta, \zeta', \zeta''$. By choice of the sequence $(\varepsilon_l)$, if for abbrevation we write $z^+ \coloneqq z + (\zeta + \zeta')$ and $z^- \coloneqq z - (\zeta + \zeta'')$, we thus get the following estimate for all $\mu > \nu$ and $j',k' \in \N_{2^\mu}$ (for suitable $j,k \in \N_{2^\nu}$ depending on $j',k'$):
\begin{equation*} \begin{split}
  & \frac{\gabs{ w_{j'}^{(\mu)}\big(z + (\zeta + \zeta') \big) - w_{k'}^{(\mu)}\big(z - (\zeta + \zeta'') \big) }}{\norm{\zeta' + 2\zeta + \zeta''}} \ge \frac{\gabs{ w_{j'}^{(\mu)}(z^+) - w_{k'}^{(\mu)}(z^-) }}{(2+C)\norm{\zeta}} \\ 
  \ge \,& \frac{\gabs{ w_j^{(\nu)}(z^+) - w_k^{(\nu)}(z^-) }}{(2 + C) \norm{\zeta}} \, - \frac{1}{(2 + C) \norm{\zeta}} \sum_{l = \nu +1}^\mu \textstyle \varepsilon_l \Big( \sqrt{\gabs{z^+_{[l]} - a_l}} + \sqrt{\gabs{z^-_{[l]} - a_l}}\, \Big) \\
  \ge \, & \displaystyle  \frac{\nu}{1 + (C/2)} - \frac{1}{(2+C)\,\delta} \sum_{l = \nu + 1}^\mu \frac{\delta}{2^{l-1}} \ge \frac{\nu - 1}{1 + (C/2)}.
\end{split} \end{equation*}
Since by Lemma \ref{thm_defE} each $(z,w) \in \mathcal{E}$ is a limit of points $\big( z, w_{j_\mu}^{(\mu)} \big)$, this proves $(\ref{equ_slopeE})$.
\end{proof}
\begin{lemma} \label{thm_analyticE}
If $\{\varepsilon_l\}$ is decreasing fast enough, then $\mathcal{E}$ contains no analytic variety of positive dimension.
\end{lemma}
\begin{proof}
Let $\{\varepsilon_l\}$ be decreasing so fast that the assertions of Lemma \ref{thm_choiceeps} hold true. To get a contradiction, assume that $\mathcal{E}$ contains an analytic variety of positive dimension. Then in particular $\mathcal{E}$ contains a nonconstant analytic disc, i.e., there exists a nonconstant holomorphic mapping $f = (f_1, f_2, \ldots, f_n) \colon \D_r(0) \to \C^n$ such that $f\big(\D_r(0)\big) \subset\mathcal{E}$, where $\D_r(\xi_0) = \{\xi \in \C : \abs{\xi-\xi_0} < r \}$. Let $P \subset \N_{n-1}$ be the set of all coordinate directions in $\C^{n-1}$ such that $f_p$ is not constant. Since by the choice of $\{\varepsilon_l\}$ and Lemma \ref{thm_defE} the set $\mathcal{E}_z$ has zero $2$-dimensional Lebesgue measure for every $z \in \C^{n-1}$, we see that $P \neq \varnothing$. Without loss of generality, we can assume that $P = \{1, \ldots, T\}$ for some $T\le n-1$. After possibly passing to a subset $\D_{r'}(\xi_0) \subset \D_r(0)$, we can assume by the implicit function theorem that there exist an open subset $U \subset \C$ and some
\[ \phi \colon U \to \C^n \text{ holomorphic, }\, \phi(U) = f\big(\D_{r'}(\xi_0)\big) \]
such that $\phi(\xi) \!=\! \big(\xi, \phi_2(\xi), \ldots, \phi_T(\xi), q_{T+1}, \ldots, q_{n-1}, \phi_n(\xi) \big) \!\eqqcolon\! \big(\phi_\ast(\xi), \phi_{n}(\xi) \big)$ with suitable constants $q_{T+1}, \ldots, q_{n-1} \in \C$. After a possible shrinking of $U$, we can assume that there exist positive numbers $\sigma, \theta > 0$ such that on $U$
\begin{equation} \label{equ_slopephi}
\theta < \abs{\phi_p'} \quad \text{for all} \,\, p \in P, \qquad \abs{\phi_p'} < \sigma \quad \text{for all} \,\, p \in \N_n.
\end{equation}
Indeed, $\theta$ exists since the zero set of each $\abs{\phi_p'}$ is discret, and we use Cauchy's estimates to find $\sigma$. Thus, after possibly shrinking $U$ again, we can assume that for $z, z' \in \phi_\ast(U)$ and $1 \le s,t \le T$ we have $\theta < \abs{z_t' - z_t} / \abs{z_1' - z_1}$ and $\abs{z_s' - z_s}/\abs{z_1'- z_1} < \sigma$, i.e., $\abs{z_s' - z_s}/\abs{z_t' - z_t} < \sigma / \theta$. In particular, we see that there exists $\alpha \coloneqq \sigma / \theta > 1$ such that
\[ D\phi_\ast(z_1)(\C) \subset \gamma(P, \alpha) \quad \text{for all} \,\, z_1 \in U. \]
Moreover (after possibly further shrinking $U$), we can assume that for every $p \in \N_n$
\begin{equation} \label{equ_restgamma}
  \frac{\gabs{\phi_p(a + \xi) - \phi_p(a) - \phi_p'(a)\xi}}{\abs{\xi}} < (C/2)\,\theta \qquad \begin{minipage}{20ex} for all $a \in U$, $\xi \in \C$ \\ such that $a + \xi \in U$. \end{minipage} 
\end{equation}
Since we can assume $f\big(\D_r(0)\big)$ to be bounded, and since $\frac{1}{9}(\varepsilon_l / \varepsilon_{l+1})^2 > l$, we can choose $\nu_0 \in \N$ so large that $\phi_\ast(U) \subset B^{n-1}(0, \nu_0)$,
\begin{equation} \label{equ_nualpha} 
  \nu_0 + 1 > \alpha \quad \text{and} \quad {\textstyle \frac{1}{9}} (\varepsilon_l / \varepsilon_{l+1})^2 > \alpha \;\; \text{for all} \;\; l \ge \nu_0. 
\end{equation}
Further, since $\max_{1 \le p \le n-1} \vol \big( \bigcup_{l \in L_\nu^p} \Delta^1(a_l, \rho_\nu) \big) \to 0$ for $\nu \to \infty$ and $\{\rho_\nu\}$ is decreasing, we can assume (after possibly enlarging $\nu_0$ and then shrinking of $U$) that $\phi_p(U) \cap \bigcup_{l \in L_{\nu_0}^p} \Delta^1(a_l, \rho_\nu) = \varnothing$ for all $p \in P$ and $\nu \ge \nu_0$. But then $\tilde{\sL}_\nu^P(z) \cap \N_{\nu_0} = \varnothing$ for all $\nu \ge \nu_0$, $z \in S_\nu \cap \phi_\ast(U)$, $[\nu] \in P$. By definition of $\tilde{\alpha}_\nu^P(z)$ and from $(\ref{equ_nualpha})$, we therefore get
\[ \textstyle \displaystyle \tilde{\alpha}_\nu^P(z) > \alpha \quad \text{for all} \;\; \nu \ge \nu_0, \, z \in S_\nu \cap \phi_\ast(U), \, [\nu] \in P. \]
After these preparations, we now choose a strictly increasing sequence $\{\nu_k\}$ of natural numbers such that for each $\nu$ from this sequence we have
\[ \nu \ge \nu_0, \quad [\nu] = 1, \quad B^{n-1}\big( \phi_\ast(a_\nu), 1/\nu \big) \subset U \times \C^{n-2}. \]
Let $\nu$ be an arbitrary fixed member of this sequence. Since $\phi_\ast(U) \subset B^{n-1}(0, \nu_0)$, we see that $z \coloneqq \phi_\ast(a_\nu) \in S_\nu \cap B^{n-1}(0, \nu)$. Hence we can use Lemma \ref{thm_choiceeps} to find a $\delta \in (0, 1/\nu)$ such that
\[ \frac{w' - w''}{\norm{\zeta' + 2\zeta + \zeta''}} \ge \frac{\nu - 1}{1 + (C/2)} \quad \text{for all} \quad w' \in \mathcal{E}_{z + (\zeta + \zeta')}, \, w'' \in \mathcal{E}_{z - (\zeta + \zeta'')} \]
and all choices of
\begin{equation} \label{equ_propzeta}
  \zeta \in \gamma\big(P, \tilde{\alpha}_\nu^P(z) \big) \cap \partial B^{n-1}(0,\delta) \quad \text{and} \quad \zeta', \zeta'' \in \Delta^{n-1}\big(0, (C/2)\abs{\zeta} \big).
\end{equation}
By the choice of $U$ and $\nu_0$, we have $D\phi_\ast(a_\nu)(\xi - a_\nu) \in \gamma(P, \alpha)$ for all $\xi \in U \setminus \{a_\nu\}$ and $\tilde{\alpha}_\nu^P(z) > \alpha$; hence $D\phi_\ast(a_\nu)(\xi - a_\nu) \in \gamma\big(P, \tilde{\alpha}_\nu^P(z) \big)$. Moreover, $B^{n-1}(z, \delta ) \subset U \times \C^{n-2}$. Thus
\begin{equation*} \begin{split}
 \Sigma \coloneqq &\big[ z + \gamma\big(P, \tilde{\alpha}_\nu^P(z) \big) \big] \\ &\cap \big[ z + \{D\phi_\ast(a_\nu)(\xi - a_\nu) : \xi \in \C \setminus \{0\} \text{ such that } a_\nu + \xi \in U \} \big] \cap \partial B^{n-1}(z, \delta)
\end{split} \end{equation*}
is nonempty. Therefore we can choose $\zeta \in \gamma\big(P, \tilde{\alpha}_\nu^P(z) \big)$ such that $z \pm \zeta \in \Sigma$, and $\xi \in \C$ such that $a_\nu \pm \xi \in U$ and $D\phi_\ast(a_\nu)(\pm\xi) = \pm \zeta$. Now applying $(\ref{equ_restgamma})$ in the case $a = a_\nu$ and using $(\ref{equ_slopephi})$ yields
\[ \gabs{\phi_p(a_\nu + \xi) - z_p - \zeta_p} < (C/2)\,\theta \abs{\xi} < (C/2)\abs{\phi_p'(a_\nu) \xi} = (C/2)\abs{\zeta_p} \]
for every $p \in P$. Since also $\phi_p(a_\nu + \xi) = z_p + \zeta_p$ for $p \in \N_{n-1} \setminus P$ and  $\phi_1(z_1) = z_1$, this shows that there exist uniquely determined $\zeta', \zeta'' \in \Delta^{n-1}\big(0, (C/2)\abs{\zeta}\big)$ such that $z + (\zeta + \zeta') = \phi_\ast(a_\nu + \xi)$, $z - (\zeta +  \zeta'') = \phi_\ast(a_\nu - \xi)$ and $\zeta_1', \zeta_1'' = 0$. In particular, we see from $\phi(U) \subset f\big(\D_{r'}(\xi_0)\big) \subset f\big(\D_r(0)\big) \subset \mathcal{E}$ that
\[ w \coloneqq \phi_n(a_\nu + \xi) \in \mathcal{E}_{z + (\zeta + \zeta')}, \qquad w' \coloneqq \phi_n(a_\nu - \xi) \in \mathcal{E}_{z - (\zeta + \zeta'')}. \]
Observe that $\zeta, \zeta' \zeta''$ satisfy the conditions in $(\ref{equ_propzeta})$. Since $a_\nu \in U$ and $\phi_1' \equiv 1$ on $U$, we get $\norm{\zeta' + 2\zeta + \zeta''} \ge (2 - C) \norm{\zeta} = (2 - C) \norm{D\phi_\ast(a_\nu)(\xi)} \ge (2 - C) \abs{\xi}$ and thus, in view of Lemma \ref{thm_choiceeps}, can finally make the following estimate:
\[ \frac{\gabs{\phi_n(a_\nu + \xi) - \phi_n(a_\nu - \xi)}}{2\abs{\xi}} \ge (1- C/2)\cdot \frac{\abs{w - w'}}{\norm{\zeta' + 2\zeta+ \zeta''}} \ge \frac{1-C/2}{1+C/2} \cdot (\nu - 1)\,. \]
This holds true for every member $\nu$ of the strictly increasing sequence $(\nu_k)$, and the right term becomes unbounded as $\nu \to + \infty$. Since for each fixed $\nu$ the number $\xi$ was chosen such that $a_\nu \pm \xi \in U$, this contradicts the fact that $\phi_n$ has a bounded derivate on $U$.
\end{proof}
%
%
%

%
%
%
%
%
%
\section{Choice of the sequence $\{\varepsilon_l\}$ - Part II}
Recall that $E_\nu = \{P_\nu = 0\}$, $\nu \in \N$. We show that for $\{\varepsilon_l\}$ decreasing fast enough we can guarantee nice convergence properties of the sequence $\{P_\nu\}$ as well as certain relations between the limit set $\mathcal{E}$ of $\{E_\nu\}$ and the sublevel sets of the defining polynomials $P_\nu$.

\begin{lemma} \label{thm_convPn}
Let $\{\varepsilon_l\}$ be chosen in such a way that $\varepsilon_l \sqrt{\abs{z_{[l]} - a_l}} < 1/2^l$ on $B^{n-1}(0,l)$ for every $l \in \N$. Then the sequence $\{\abs{P_\nu}^{1/2^\nu}\}$ converges uniformly on compact subsets of $\C^n \setminus \mathcal{E}$, and $\lim_{\nu \to \infty} \abs{P_\nu}^{1/2^\nu} > 0$ on $\C^n \setminus \mathcal{E}$.
\end{lemma}
\begin{proof}
Fix $(z_0,w_0) \in \C^n \setminus \mathcal{E}$ and choose $R > 0$ such that $(z_0,w_0) \in \Delta_R \coloneqq B^{n-1}(0,R) \times \C$. Since $\mathcal{E}$ is closed and $E_\nu \cap \overline{\Delta}_R \to \mathcal{E} \cap \overline{\Delta}_R$ in the Hausdorff metric, there exist a ball $B \coloneqq B^n\big((z_0,w_0), \delta\big) \subset \Delta_R$ and positive numbers $r > 0$, $N_r > 0$ such that $\dist (B, E_\nu) > r$ for all $\nu \ge N_r$. Now for every $\nu, \mu \in \N$, $j \in \N_{2^\nu}$ and $z \in \C^{n-1}$ we denote the $2^\mu$ values of $w_j^{(\nu)}(z) + \sum_{l = \nu + 1}^{\nu + \mu} \varepsilon_l \sqrt{z_{[l]} - a_l}$ by $w_1^{(\mu)}(\nu, j; z), \ldots, w_{2^\mu}^{(\mu)}(\nu, j; z)$. Observe that with this notation we have
\[ \abs{P_{\nu+\mu}(z,w)}^{1/2^{\nu+\mu}} \!\!= \! \prod_{l = 1}^{2^{\nu+\mu}} \abs{w - w_{l}^{(\nu+\mu)}(z)}^{1/2^{\nu+\mu}} \!\!=  \prod_{j = 1}^{2^\nu} \prod_{k=1}^{2^\mu} \abs{w - w_k^{(\mu)}(\nu,j;z)}^{1/2^{\nu+\mu}}; \]
thus passing from $\abs{P_\nu(z,w)}^{1/2^\nu}$ to $\abs{P_{\nu+\mu}(z,w)}^{1/2^{\nu+\mu}}$ amounts to replace each term $\abs{w - w_j^{(\nu)}(z)}$ occuring in the product expansion of $\abs{P_\nu(z,w)}^{1/2^\nu}$ by the mean value $\prod_{k=1}^{2^\mu} \abs{w - w_k^{(\mu)}(\nu,j;z)}^{1/2^\mu}$. Since for $\nu \ge R$ one has $\abs{w_j^{(\nu)}(z) - w_k^{(\mu)}(\nu,j;z)} \le \sum_{l=\nu+1}^{\nu+\mu} \varepsilon_l \sqrt{\abs{z_{[l]} - a_l}} < 1/2^\nu$ for all $z \in B^{n-1}(0,R)$, we can estimate the resulting error, by means of
\[ \prod_{k=1}^{2^\mu} \abs{w - w_k^{(\mu)}(\nu,j;z)}^{1/2^\mu} > \prod_{k=1}^{2^\mu} \Big( \abs{w-w_j^{(\nu)}(z)} - 1/2^\nu \Big)^{1/2^\mu} \!\!\!\! = \abs{w-w_j^{(\nu)}(z)} - 1/2^\nu, \]
%
%
\[ \prod_{k=1}^{2^\mu} \abs{w - w_k^{(\mu)}(\nu,j;z)}^{1/2^\mu} < \prod_{k=1}^{2^\mu} \Big( \abs{w-w_j^{(\nu)}(z)} + 1/2^\nu \Big)^{1/2^\mu} \!\!\!\! = \abs{w-w_j^{(\nu)}(z)} + 1/2^\nu, \]
to be less than $1/2^\nu$ for all $(z,w) \in B \subset B^{n-1}(0,R) \times \C$ (obviously the first inequality is trivial if $\abs{w-w_j^{(\nu)}(z)} < 1/2^\nu$). In particular, whenever $\abs{w - w_j^{(\nu)}(z)} \ge 1/2^\nu$ on $B$ and $\nu \ge R$, we get
\begin{equation*} \label{equ_CauchyPn}
\prod_{j = 1}^{2^\nu} \Big(\abs{w-w_j^{(\nu)}(z)} - 1/2^\nu\Big)^{1/2^\nu} \!\! \le \abs{P_{\nu+\mu}(z,w)}^{1/2^{\nu+\mu}} \! \le \prod_{j = 1}^{2^\nu} \Big(\abs{w-w_j^{(\nu)}(z)} + 1/2^\nu\Big)^{1/2^\nu}
\end{equation*}
on $B$. But $\abs{w - w_j^{(\nu)}(z)} > r$ on $B$ for all $\nu \ge N_r$, where $r$ does not depend on $\nu$. Since $\abs{P_\nu(z,w)} = \prod_{j=1}^{2^\nu}\abs{w-w_j^{(\nu)}(z)}$, this shows that $\big\{\abs{P_\nu(z,w)}^{1/2^\nu}\big\}_{\nu \ge 1}$ is a Cauchy sequence for every $(z,w) \in B$ and in fact that $\big\{\abs{P_\nu}^{1/2^\nu}\big\}_{\nu \ge 1}$ converges uniformly on $B$. Moreover, $\lim_{\nu \to \infty} \abs{P_\nu}^{1/2^\nu} > 0$ on $B$, since the above estimates hold true for all $\mu \in \N$.
\end{proof}
\begin{lemma} \label{thm_sublevelE}
If $\{\varepsilon_l\}$ is decreasing fast enough, then
\begin{equation} \label{equ_sublevelE}
\mathcal{E} = \bigcap_{\nu \in \N}\bigcup_{\mu \ge \nu} \big\{ \abs{P_\mu} < (\textstyle\frac{1}{\mu})^{2^\mu} \big\}.
\end{equation}
Moreover, the following relations hold true for every $\mu \ge \nu \ge R$:
\begin{enumerate}
  \item[(1)] $\{ \abs{P_\mu} < (\frac{1}{\nu+1})^{2^\mu} \} \cap \Bbarn(0,R) \subset \subset \{ \abs{P_\nu} < (\frac{1}{\nu})^{2^\nu} \}.$
  \item[(2)] $\{ \abs{P_\nu} < (\frac{1}{\nu})^{2^\nu} \} \cap \Bbarn(0,R) \subset \subset \{ \abs{P_\mu} < (\frac{1}{\nu-1})^{2^\mu} \}.$
\end{enumerate}
\end{lemma}
\begin{proof}
For $M \subset \C^n$ and $R, \delta > 0$ we let $M_{(R)} \coloneqq M \cap \Bbarn(0,R)$ and
\[ M^{\langle \delta \rangle} \coloneqq M \cup \bigcup_{x \in \partial M} B^n(x,\delta) \quad \text{and} \quad M^{\langle -\delta \rangle} \coloneqq M \setminus \bigcup_{x \in \partial M} B^n(x,\delta). \]
One easily verifies the following relations for all $M,N \subset \C^n$ and $R,\delta,\delta_1,\delta_2 > 0$:
\begin{enumerate}
\item[(A)] $M \subset N \Rightarrow M^{\langle\delta\rangle} \subset N^{\langle\delta\rangle}$ \;and\; $M \subset N \Rightarrow M^{\langle-\delta\rangle} \subset N^{\langle-\delta\rangle}$.
\item[(B)] $M_{(R)} \subset N \Rightarrow M_{(R)} \subset \subset N^{\langle \delta \rangle}$ \;and\; $M_{(R)} \subset N$ $\Rightarrow [M^{\langle -\delta \rangle}]_{(R)} \subset \subset N$.
\item[(C)] $[ M^{\langle \delta_1 \rangle} ]^{\langle \delta_2 \rangle} = M^{\langle \delta_1 + \delta_2 \rangle}$ \;and\; $[ M^{\langle -\delta_1 \rangle} ]^{\langle -\delta_2\rangle} = M^{\langle -(\delta_1 + \delta_2) \rangle}$.
\item[(D)] $[M^{\langle \delta \rangle}]_{(R-\delta)} \subset [M_{(R)}]^{\langle \delta \rangle}$ \;and\; $[M^{\langle -\delta \rangle}]_{(R-\delta)} \subset [M_{(R)}]^{\langle -\delta \rangle}$.
\end{enumerate}
Moreover, $M^{\langle \pm \delta \rangle}_{(R)}$ will denote the set $M^{\langle \pm \delta \rangle} \cap \Bbarn(0,R)$. We can choose sequences $\{\varepsilon_l\}, \{\delta_l\}$ of positive numbers converging to zero such that for all $\nu \in \N$ the following relations hold true: \vspace{-1ex}
\[ \renewcommand{\arraystretch}{1.75} \begin{array}{cl} 
  (1_\nu) & \varepsilon_\nu \sqrt{\abs{z_{[\nu]} - a_\nu}} < \frac{1}{2^\nu} \text{ on } B^{n-1}(0,\nu). \\
  (2_\nu) & \Big[ \big\{ \abs{P_\nu} < (\frac{1}{\nu+1})^{2^\nu} \big\}_{(\nu+1)} \cup \big\{ \abs{P_\nu} > (\frac{1}{\nu-1})^{2^\nu} \big\}_{(\nu+1)}  \Big] \cap \big\{ \abs{P_\nu} = (\frac{1}{\nu})^{2^\nu} \big\}^{\langle \delta_\nu \rangle} = \varnothing .\\
  (3_{\nu+1}) & \big\{ \abs{P_{\nu+1}} < (\frac{1}{\lambda})^{2^{\nu+1}} \big\}_{(\nu + 1)} \subset \big\{ \abs{P_\nu} < (\frac{1}{\lambda})^{2^\nu} \big\}^{\langle \delta_\nu/2^\nu \rangle} \quad \text{for } \lambda = 1, \ldots, \nu+1. \\
   (3_{\nu+1}') & \big\{ \abs{P_\nu} < (\frac{1}{\lambda})^{2^\nu} \big\}^{\langle -\delta_\nu/2^\nu \rangle}_{(\nu + 1)} \subset \big\{ \abs{P_{\nu+1}} < (\frac{1}{\lambda})^{2^{\nu+1}} \big\} \quad \text{for } \lambda = 1, \ldots, \nu - 1. \\
\end{array} \renewcommand{\arraystretch}{1}\]
Indeed, we can choose $\varepsilon_1$ to satisfy $(1_1)$. After fixing such $\varepsilon_1$, the polynomial $P_1$ is fixed, and we can choose $\delta_1 < 1/2$ to satisfy $(2_1)$. Suppose now that $\varepsilon_l, \delta_l$ are already chosen for $l = 1, 2, \ldots, \nu$ such that $(1_\nu)$-$(3_\nu')$ hold true. By Lemma \ref{thm_defPn} we know that $P_{\nu+1} \rightarrow P_\nu^2$ uniformly on compact subsets as $\varepsilon_{\nu+1} \rightarrow 0$; hence we can find $\varepsilon > 0$ such that for $\varepsilon_{\nu+1} < \varepsilon$ the polynomial $P_{\nu+1}$ satisfies $(3_{\nu+1})$ and $(3_{\nu+1}')$. Moreover, we can find $\varepsilon' > 0$ such that for $\varepsilon_{\nu+1} < \varepsilon'$ the inequality $(1_{\nu+1})$ holds true. We choose $\varepsilon_{\nu+1} < \min \{ \varepsilon, \varepsilon' \}$, and we point out that every $\varepsilon_{\nu+1}'\in(0,\varepsilon_{\nu+1})$ would also be a proper choice for $\varepsilon_{\nu+1}$. For $P_{\nu+1}$ now being fixed, we can find $\delta_{\nu+1} < \delta_\nu / 2$ satisfying $(2_{\nu+1})$.

\noindent (i) We now prove statement (1) of the lemma. In order to do this we need the following

\noindent \textsc{Claim 1}. For $\mu > \nu \ge R$, one has

\[ \big\{ \abs{P_\mu} < (\textstyle\frac{1}{\nu+1})^{2^\mu} \big\}_{(R)} \subset \big\{ \abs{P_\nu} < (\textstyle\frac{1}{\nu+1})^{2^\nu} \big\}^{\langle \sum_{l=\nu}^{\mu-1} \delta_l / 2^l \rangle}. \]
\textsc{Proof}. Let $\mu > \nu \ge R$ be fixed. For proving the statement of the claim, we use reverse induction on $\rho$ to show that
\begin{equation} \label{equ_sublevelind+1}
 \big\{ \abs{P_\mu} < (\textstyle\frac{1}{\nu+1})^{2^\mu} \big\}_{(R)} \subset \big\{ \abs{P_\rho} < (\textstyle\frac{1}{\nu+1})^{2^\rho} \big\}^{\langle \sum_{l=\rho}^{\mu-1} \delta_l / 2^l \rangle} \quad \text{for } \rho = \mu-1, \ldots, \nu. 
\end{equation}
The case $\rho = \mu-1$ follows immediately from $(3_\mu)$ with $\lambda = \nu + 1$. Suppose that property $(\ref{equ_sublevelind+1})$ holds for some $\rho \in \N$ such that $\mu > \rho > \nu \ge R$. Then one also has
\begin{equation} \label{equ_sublevelind+2}
\big\{ \abs{P_\mu} < (\textstyle\frac{1}{\nu+1})^{2^\mu} \big\}_{(R)} \subset \big\{ \abs{P_\rho} < (\textstyle\frac{1}{\nu+1})^{2^\rho} \big\}^{\langle \sum_{l=\rho}^{\mu-1} \delta_l / 2^l \rangle}_{(R)}.
\end{equation}
Hence applying $(3_\rho)$ with $\lambda = \nu + 1$, we can conclude that
\begin{align*}
& \big\{ \abs{P_\rho} < (\textstyle \frac{1}{\nu+1} \displaystyle)^{2^\rho} \big\}_{(\rho)} \subset \big\{ \abs{P_{\rho-1}} < (\textstyle \frac{1}{\nu+1} \displaystyle)^{2^{\rho-1}} \big\}^{\langle \delta_{\rho-1}/2^{\rho-1} \rangle} \\
\Rightarrow & \big[\big\{ \abs{P_\rho} < (\textstyle \frac{1}{\nu+1} \displaystyle)^{2^\rho} \big\}_{(\rho)}\big]^{\langle \sum_{l=\rho}^{\mu-1} \delta_l/2^l \rangle} \subset \big[ \big\{ \abs{P_{\rho-1}} < (\textstyle \frac{1}{\nu+1} \displaystyle)^{2^{\rho-1}} \big\}^{\langle \delta_{\rho-1}/2^{\rho-1} \rangle} \big]^{\langle \sum_{l=\rho}^{\mu-1} \delta_l/2^l \rangle} \\
\Rightarrow & \big[\big\{ \abs{P_\rho} < (\textstyle \frac{1}{\nu+1} \displaystyle)^{2^\rho} \big\}^{\langle \sum_{l=\rho}^{\mu-1} \delta_l/2^l \rangle} \big]_{(\rho - \sum_{l=\rho}^{\mu-1} \delta_l/2^l)} \subset \big\{ \abs{P_{\rho-1}} < (\textstyle \frac{1}{\nu+1} \displaystyle)^{2^{\rho-1}} \big\}^{\langle \sum_{l=\rho-1}^{\mu-1} \delta_l/2^l \rangle} \\
\Rightarrow & \big[\big\{ \abs{P_\rho} < (\textstyle \frac{1}{\nu+1} \displaystyle)^{2^\rho} \big\}^{\langle \sum_{l=\rho}^{\mu-1} \delta_l/2^l \rangle} \big]_{(R)} \subset \big\{ \abs{P_{\rho-1}} < (\textstyle \frac{1}{\nu+1} \displaystyle)^{2^{\rho-1}} \big\}^{\langle \sum_{l=\rho-1}^{\mu-1} \delta_l/2^l \rangle}.
\end{align*}
This, together with $(\ref{equ_sublevelind+2})$, completes our argument by induction and proves Claim 1.$\Box$

Observe that, since $\{\delta_l\}$ is monotonically decreasing, we get from Claim 1 and $(B)$ the following property:
\begin{equation} \label{equ_sublevel+1} 
\big\{ \abs{P_\mu} < (\textstyle\frac{1}{\nu+1})^{2^\mu} \big\}_{(R)} \subset \subset \big\{ \abs{P_\nu} < (\textstyle\frac{1}{\nu+1})^{2^\nu} \big\}^{\langle \delta_\nu \rangle}.
\end{equation}
Fix now some $\nu \ge R$. We are going to show that 
\begin{equation} \label{equ_sublevel+2} 
\big\{ \abs{P_\nu} < (\textstyle\frac{1}{\nu+1}\displaystyle)^{2^\nu} \big\}^{\langle \delta_\nu \rangle}_{(R)} \subset \big\{ \abs{P_\nu} < (\textstyle \frac{1}{\nu} \displaystyle)^{2^\nu} \big\}. 
\end{equation}
Note that $(\ref{equ_sublevel+1})$ and $(\ref{equ_sublevel+2})$ together prove (1). By definition, we have
\[ \big\{ \abs{P_\nu} < (\textstyle\frac{1}{\nu+1}\displaystyle)^{2^\nu} \big\}^{\langle \delta_\nu \rangle}_{(R)} = \big\{ \abs{P_\nu} < (\textstyle\frac{1}{\nu+1}\displaystyle)^{2^\nu} \big\}_{(R)} \cup \bigcup_{x \in \partial \{\abs{P_\nu} < (\frac{1}{\nu+1})^{2^\nu}\}} B^n(x,\delta_\nu)_{(R)} \]
Obviously 
\[ \big\{ \abs{P_\nu} < (\textstyle\frac{1}{\nu+1})^{2^\nu} \big\}_{(R)} \subset \big\{ \abs{P_\nu} < (\textstyle \frac{1}{\nu} \displaystyle)^{2^\nu} \big\}. \]
Let $\zeta \in B^n(x, \delta_\nu)_{(R)}$ for some $x \in \partial \{\abs{P_\nu} < (\frac{1}{\nu+1})^{2^\nu}\}$. Then in particular $x \in \{\abs{P_\nu} = (\frac{1}{\nu+1})^{2^\nu}\}_{(\nu+1)}$. Assume, to get a contradiction, that $\zeta \in \{ \abs{P_\nu} \ge (\frac{1}{\nu})^{2^\nu} \}$. Since $x \in \{\abs{P_\nu} < (\frac{1}{\nu})^{2^\nu}\}_{(\nu+1)}$, we then can find $t \in (0,1]$ such that $\tilde{x} \coloneqq (1-t)x + t\zeta \in \{\abs{P_\nu} = (\frac{1}{\nu})^{2^\nu}\}$. Now obviously $\norm{\tilde{x} - x} < \delta_\nu$, which shows that $x \in \{\abs{P_\nu} = (\frac{1}{\nu+1})^{2^\nu}\}_{(\nu+1)} \cap \{\abs{P_\nu} = (\frac{1}{\nu})^{2^\nu}\}^{\langle \delta_\nu \rangle}$. In particular, we conclude that $\{\abs{P_\nu} < (\frac{1}{\nu+1})^{2^\nu}\}_{(\nu+1)} \cap \{\abs{P_\nu} = (\frac{1}{\nu})^{2^\nu}\}^{\langle \delta_\nu \rangle} \neq \varnothing$, which contradicts $(2_\nu)$. This proves that
\[ \bigcup_{x \in \partial \{\abs{P_\nu} < (\frac{1}{\nu+1})^{2^\nu}\}} B^n(x,\delta_\nu)_{(R)} \subset \big\{ \abs{P_\nu} < (\textstyle \frac{1}{\nu} \displaystyle)^{2^\nu} \big\},\]
and hence $(\ref{equ_sublevel+2})$. The proof of statement (1) of the lemma is now complete.

\noindent (ii) We now prove statement (2) of the lemma. For being able to do this we need the following

\noindent \textsc{Claim 2}. For $\mu > \nu \ge R$, one has 
%
\[ \big\{ \abs{P_\nu} < (\textstyle\frac{1}{\nu-1})^{2^\nu} \big\}^{\langle -\sum_{l=\nu}^{\mu-1} \delta_l / 2^l \rangle}_{(R)} \subset \big\{ \abs{P_\mu} < (\textstyle\frac{1}{\nu-1}\displaystyle)^{2^\mu} \big\}. \]
%
\textsc{Proof}. Let $\mu > \nu \ge R$ be fixed. For proving the statement of the claim, we use induction on $\rho$ to show that
\begin{equation} \label{equ_sublevelind-}
\big\{ \abs{P_\nu} < (\textstyle\frac{1}{\nu-1})^{2^\nu} \big\}^{\langle -\sum_{l=\nu}^{\rho-1} \delta_l / 2^l \rangle}_{(\nu + 1 - \sum_{l=\nu}^{\rho-1} \delta_l / 2^l)} \subset \big\{ \abs{P_\rho} < (\textstyle\frac{1}{\nu-1}\displaystyle)^{2^\rho} \big\}, \quad \text{for } \rho = \nu+1, \ldots, \mu.
\end{equation}
The case $\rho = \nu + 1$ follows immediately from $(3_{\nu+1}')$ with $\lambda = \nu - 1$. Suppose that property $(\ref{equ_sublevelind-})$ holds for some $\rho \in \N$ such that $\mu > \rho > \nu \ge R$. Then we also have
\begin{align*}
& \big[ \big\{ \abs{P_\nu} < (\textstyle\frac{1}{\nu-1}\displaystyle)^{2^\nu} \big\}^{\langle -\sum_{l=\nu}^{\rho-1} \delta_l / 2^l \rangle} \big]_{(\nu + 1 - \sum_{l=\nu}^{\rho-1} \delta_l / 2^l)} \subset \big\{ \abs{P_\rho} < (\textstyle\frac{1}{\nu-1}\displaystyle)^{2^\rho} \big\} \\
\Rightarrow & \Big[ \big[ \big\{ \abs{P_\nu} < (\textstyle\frac{1}{\nu-1}\displaystyle)^{2^\nu} \big\}^{\langle -\sum_{l=\nu}^{\rho-1} \delta_l / 2^l \rangle} \big]_{(\nu + 1 - \sum_{l=\nu}^{\rho-1} \delta_l / 2^l)} \Big]^{\langle -\delta_\rho/2^\rho \rangle} \!\!\!\!\! \subset \big\{ \abs{P_\rho} < (\textstyle\frac{1}{\nu-1}\displaystyle)^{2^\rho} \big\}^{\langle -\delta_\rho/2^\rho \rangle} \\
\Rightarrow & \big[ \big\{ \abs{P_\nu} < (\textstyle\frac{1}{\nu-1}\displaystyle)^{2^\nu} \big\}^{\langle -\sum_{l=\nu}^\rho \delta_l / 2^l \rangle} \big]_{(\nu + 1 - \sum_{l=\nu}^\rho \delta_l / 2^l)} \subset \big\{ \abs{P_\rho} < (\textstyle\frac{1}{\nu-1}\displaystyle)^{2^\rho} \big\}^{\langle -\delta_\rho/2^\rho \rangle} \\
\Rightarrow & \big[ \big\{ \abs{P_\nu} < (\textstyle\frac{1}{\nu-1}\displaystyle)^{2^\nu} \big\}^{\langle -\sum_{l=\nu}^\rho \delta_l / 2^l \rangle} \big]_{(\nu + 1 - \sum_{l=\nu}^\rho \delta_l / 2^l)} \subset \big\{ \abs{P_\rho} < (\textstyle\frac{1}{\nu-1}\displaystyle)^{2^\rho} \big\}^{\langle -\delta_\rho/2^\rho \rangle}_{(\nu+1)}
\end{align*}
while from $(3_{\rho+1}')$ with $\lambda = \nu - 1$, we get
\[ \big\{ \abs{P_\rho} < (\textstyle\frac{1}{\nu-1}\displaystyle)^{2^\rho} \big\}^{\langle -\delta_\rho/2^\rho \rangle}_{(\nu+1)} \subset \big\{ \abs{P_{\rho+1}} < (\textstyle\frac{1}{\nu-1}\displaystyle)^{2^{\rho+1}} \big\}. \]
This completes our argument by induction and, since $\nu + 1 - \sum_{l=\nu}^{\mu-1} \delta_l/2^l > R$, proves Claim 2. $\Box$

Observe that, since $\{\delta_l\}$ is monotonically decreasing, we get from Claim 2 and $(B)$ the following property:
\begin{equation} \label{equ_sublevel-1} \big\{ \abs{P_\nu} < (\textstyle\frac{1}{\nu-1}\displaystyle)^{2^\nu} \big\}^{\langle -\delta_\nu \rangle}_{(R)} \subset \subset \big\{ \abs{P_\mu} < (\textstyle\frac{1}{\nu-1}\displaystyle)^{2^\mu} \big\}. \end{equation}
Fix now some $\nu \ge R$. We are going to show that
\begin{equation} \label{equ_sublevel-2} \big\{ \abs{P_\nu} < (\textstyle\frac{1}{\nu}\displaystyle)^{2^\nu} \big\}_{(R)} \subset \big\{ \abs{P_\nu} < (\textstyle\frac{1}{\nu-1}\displaystyle)^{2^\nu} \big\}^{\langle -\delta_\nu \rangle}_{(R)}. \end{equation}
Note that $(\ref{equ_sublevel-1})$ and $(\ref{equ_sublevel-2})$ together prove (2). By definition, we have
\[ \big\{ \abs{P_\nu} < (\textstyle\frac{1}{\nu-1}\displaystyle)^{2^\nu} \big\}^{\langle -\delta_\nu \rangle}_{(R)} = \big\{ \abs{P_\nu} < (\textstyle\frac{1}{\nu-1}\displaystyle)^{2^\nu} \big\}_{(R)} \setminus \bigcup_{x \in \partial \{\abs{P_\nu} < (\frac{1}{\nu-1})^{2^\nu}\}} B^n(x,\delta_\nu)_{(R)}. \]
Obviously
\[ \big\{ \abs{P_\nu} < (\textstyle\frac{1}{\nu}\displaystyle)^{2^\nu} \big\}_{(R)} \subset \big\{ \abs{P_\nu} < (\textstyle\frac{1}{\nu-1}\displaystyle)^{2^\nu} \big\}_{(R)}. \]
Let $\zeta \in B^n(x,\delta_\nu)_{(R)}$ for some $x \in \partial \{\abs{P_\nu} < (\frac{1}{\nu-1})^{2^\nu}\}$. Then in particular $x \in \{\abs{P_\nu} = (\frac{1}{\nu-1})^{2^\nu}\}_{(\nu+1)}$. In order to get a contradiction, assume that $\zeta \in \{\abs{P_\nu} < (\frac{1}{\nu})^{2^\nu}\}_{(R)}$. Since $x \in \{ \abs{P_\nu} > (\frac{1}{\nu})^{2^\nu}\}$, we then can find $t \in (0,1)$ such that $\tilde{x} \coloneqq (1-t)x + t\zeta \in \{ \abs{P_\nu} = (\frac{1}{\nu})^{2^\nu}\}$. Now obviously $\norm{\tilde{x} - x} < \delta_\nu$, which shows that $x \in \{\abs{P_\nu} = (\frac{1}{\nu-1})^{2^\nu}\}_{(\nu+1)} \cap \{\abs{P_\nu} = (\frac{1}{\nu})^{2^\nu}\}^{\langle \delta_\nu \rangle}$. In particular, we conclude that $\{\abs{P_\nu} > (\frac{1}{\nu-1})^{2^\nu}\}_{(\nu+1)} \cap \{\abs{P_\nu} = (\frac{1}{\nu})^{2^\nu}\}^{\langle \delta_\nu \rangle} \neq \varnothing$, which contradicts $(2_\nu)$. This proves that
\[ \big\{ \abs{P_\nu} < (\textstyle\frac{1}{\nu}\displaystyle)^{2^\nu} \big\}_{(R)} \cap \bigcup_{x \in \partial \{\abs{P_\nu} < (\frac{1}{\nu-1})^{2^\nu}\}} B^n(x,\delta_\nu)_{(R)} = \varnothing,\]
and hence $(\ref{equ_sublevel-2})$. The proof of statement (2) of the lemma is now complete.

Finally, we show that the representation $(\ref{equ_sublevelE})$ holds true. Let $(z,w) \in \C^n$ and choose $R > 0$ such that $(z,w) \in B^n(0,R)$. Assume that $(z,w) \in \mathcal{E}$. Let $\mu \ge R$. Applying $(1)$, we get
\[ (E_{\mu+l})_{(R)} \subset \big\{ \abs{P_{\mu+l}} < (\textstyle\frac{1}{\mu+l})^{2^{\mu+l}} \big\}_{(R)} \subset \subset \big\{ \abs{P_\mu} < (\textstyle\frac{1}{\mu})^{2^\mu} \big\} \]
for all $l \in \N$. But since $(1_\rho)$ holds true for all $\rho \in \N$, we can apply Lemma \ref{thm_defE} to see that $\mathcal{E}_{(R)} = \lim_{l \to \infty} (E_{\mu+l})_{(R)}$ in the Hausdorff metric. Hence $\mathcal{E}_{(R)} \subset \{ \abs{P_{\mu+1}} \le (\textstyle\frac{1}{\mu+1})^{2^{\mu+1}} \}_{(R)} \subset \{ \abs{P_\mu} < (\textstyle\frac{1}{\mu})^{2^\mu} \}$. Since this holds true for all $\mu \ge R$, it follows $(z,w) \in \bigcap_{\nu \in \N} \bigcup_{\mu \ge \nu} \{ \abs{P_\mu} < (\frac{1}{\mu})^{2^\mu} \}$. Conversely, assume that $(z,w) \notin \mathcal{E}$. Then by Lemma \ref{thm_convPn}, the sequence $\{\abs{P_\nu(z,w)}^{1/2^\nu}\}$ is converging to a positive real number; hence there exist $\delta > 0$ and $\mu_0 \in \N$ such that $\abs{P_\mu(z,w)}^{1/2^\mu} > \delta$ for all $\mu \ge \mu_0$. In particular, $(z,w) \notin \{ \abs{P_\mu} < (\frac{1}{\mu})^{2^\mu} \}$ for $\mu \ge \max \{\mu_0, 1/\delta \}$, which shows that $(z,w) \notin \bigcap_{\nu \in \N} \bigcup_{\mu \ge \nu} \{ \abs{P_\mu} < (\frac{1}{\mu})^{2^\mu} \}$.
\end{proof}
%
%
%

%
%
%
%
%
%
\section{Proofs of the theorems. Open questions}
We now fix the sequence $\{\varepsilon_l\}$ once and for all to be converging to zero so fast that the conclusions of Lemma \ref{thm_analyticE} and \ref{thm_sublevelE} hold true and that $\varepsilon_l \sqrt{\abs{z_{[l]} - a_l}} < 1/2^l$ on \nolinebreak $B^{n-1}(0,l)$. 

\noindent For each $\nu \in \N$, define a function $\varphi_\nu \colon \C^n \to [-\infty, +\infty)$ as 
\[ \varphi_\nu(z,w) \coloneqq \frac{1}{2^\nu} \log \abs{P_\nu(z,w)}. \]
Then $\varphi_\nu$ is a plurisubharmonic function in $\C^n$, pluriharmonic in $\C^n \setminus E_\nu$, and $\varphi_\nu(z,w) = -\infty$ if and only if $(z,w) \in E_\nu$. 

\begin{lemma} \label{thm_convphi}
The sequence $\{\varphi_\nu\}$ converges uniformly on compact subsets of $\C^n \setminus \mathcal{E}$ to a pluriharmonic function  $\varphi \colon \C^n \setminus \mathcal{E} \to \R$, and $\lim_{(z,w) \to (z_0,w_0)} \varphi(z,w) = -\infty$ for every $(z_0,w_0) \in \mathcal{E}$. In particular, $\varphi$ has an unique extension to a plurisubharmonic function on $\C^n$.
\end{lemma}
\begin{proof}
Applying Lemma $\ref{thm_convPn}$, we immediately see that $\{\varphi_\nu\}$ converges uniformly on compact subsets of $\C^n \setminus \mathcal{E}$. In particular, $\varphi$ is pluriharmonic in $\C^n \setminus \mathcal{E}$. Let $(z_0,w_0) \in \mathcal{E}$ and let $\{(z_j,w_j)\}_{j \ge 1}$ be an arbitrary sequence of points converging to $(z_0, w_0)$. Let $R \in \N$ be  such that $(z_0,w_0) \in B^n(0,R)$. From part $(1)$ of Lemma $\ref{thm_sublevelE}$ we know that 
\[ \{ \abs{P_{\mu + 1}} < (\textstyle \frac{1}{\mu+1} \displaystyle )^{2^{\mu+1}} \} \cap \Bbarn(0,R) \subset \{ \abs{P_\mu} < (\textstyle \frac{1}{\mu} \displaystyle )^{2^\mu} \} \]
for every $\mu \ge R$; thus it follows from
\[ \mathcal{E} = \bigcap_{\nu \in \N}\bigcup_{\mu \ge \nu} \big\{ \abs{P_\mu} < (\textstyle\frac{1}{\mu})^{2^\mu} \big\} \]
that $\mathcal{E} \cap \Bbarn(0,R) \subset \{ \abs{P_\nu} < (\frac{1}{\nu})^{2^\nu} \}$ for all $\nu \ge R$. Hence for every $\nu \ge R$ there exists $j(\nu) \in \N$ such that $(z_j,w_j) \in \{ \abs{P_\nu} < (\frac{1}{\nu})^{2^\nu} \} \cap B^n(0,R)$ for all $j \ge j(\nu)$. But whenever $(z_j,w_j) \in \{ \abs{P_\nu} < (\frac{1}{\nu})^{2^\nu} \} \cap \Bbarn(0,R)$ we know from part $(2)$ of Lemma $\ref{thm_sublevelE}$ that also $(z_j,w_j) \in \{ \abs{P_\mu} < (\frac{1}{\nu-1})^{2^\mu} \}$ for each $\mu \ge \nu$. This means that $\varphi_\mu(z_j,w_j) < - \log(\nu-1)$ for each $\mu \ge \nu$. Hence $\varphi(z_j, w_j) \le - \log(\nu-1)$ for each $j \ge j(\nu)$. This shows that $\lim_{j \to \infty} \varphi(z_j, w_j) = -\infty$. 
\end{proof}

\noindent {\bf Proof of Theorem 1.} By construction we have $\mathcal{E} =  \{z \in \C^n : \varphi(z) = - \infty\}$, and $\varphi$ is pluriharmonic in $\C^n \setminus \mathcal{E}$ by Lemma \ref{thm_convphi}. Using the representation $(\ref{equ_sublevelE})$ of $\mathcal{E}$ by sublevel sets of the polynomials $P_\nu$, we get
\[ \C^n \setminus \mathcal{E} = \bigcup_{\nu \in \N} \bigcap_{\mu \ge \nu} \{ \varphi_\mu \ge - \log \mu \}; \]
hence $\C^n \setminus \mathcal{E}$ is pseudoconvex. It only remains to show that $\widehat{\partial B^n(0,R) \cap \mathcal{E}} = \Bbarn(0,R) \cap \mathcal{E}$. Using $(\ref{equ_sublevelE})$ and part (1) of Lemma $\ref{thm_sublevelE}$, we see that for every $(z,w) \in \C^n \setminus \mathcal{E}$ there exists $\nu \in \N$ such that $\Bbarn(0,R) \cap \mathcal{E} \subset \{\abs{P_\nu} < (\frac{1}{\nu})^{2^\nu}\}$ but $\abs{P_\nu(z,w)} \ge (\frac{1}{\nu})^{2^\nu}$, i.e., $(z,w) \notin \widehat{\partial B^n(0,R) \cap \mathcal{E}}$. However, since clearly $\widehat{\partial B^n(0,R) \cap \mathcal{E}} \subset \Bbarn(0,R)$, this shows that $\widehat{\partial B^n(0,R) \cap \mathcal{E}} \subset \Bbarn(0,R) \cap \mathcal{E}$. Concerning the other direction, note that $\widehat{\partial B^n(0,R) \cap E_\nu} = \Bbarn(0,R) \cap E_\nu$ for every $\nu \in \N$ by the maximum modulus principle and the fact that $E_\nu$ is the zero set of the polynomial $P_\nu$. Since on bounded subsets of $\C^n$ the sequence $\{E_\nu\}$ converges to $\mathcal{E}$ in the Hausdorff metric, we thus conclude that $\Bbarn(0,R) \cap \mathcal{E} = \lim_{\nu \to \infty} \Bbarn(0,R) \cap E_\nu = \lim_{\nu \to \infty} \widehat{\partial B^n(0,R)\cap E_\nu} \subset \widehat{\partial B^n(0,R) \cap \mathcal{E}}$. \hfill $\Box$ \bigskip

\noindent {\bf Proof of Theorem 2.} For each $C_1\in\R$, we define $\Omega_{C_1}\subset \C^n$ to be the domain
\[ \Omega_{C_1} \coloneqq \big\{ (z,w) \in \C^n : \varphi(z,w) + \big( \norm{z}^2 + \abs{w}^2 \big) < C_1 \big\}, \]
where $\varphi(z,w)$ is the function constructed in Lemma \ref{thm_convphi}. It follows from the plurisubharmonicity of $\varphi$ on $\C^n$ that $\Omega_{C_1}$ is strictly pseudoconvex. Obviously one also has that $\mathcal{E} = \{\varphi = - \infty\} \subset \Omega_{C_1}$. Further, by Sard's theorem, we can choose a constant $C_1$ such that $\Omega_{C_1}$ has $C^\infty$-smooth boundary. We fix such a constant $C_1$ and define $\Omega$ to be the domain $\Omega_{C_1}$. By construction, $\mathcal{E}$ contains no analytic variety of positive dimension. Using the representation $(\ref{equ_sublevelE})$ of $\mathcal{E}$ by sublevel sets of the polynomials $P_\nu$, we get
\[ \Omega \setminus \mathcal{E} = \bigcup_{\nu \in \N} \bigcap_{\mu \ge \nu} \Big( \Omega \cap \{ \varphi_\mu \ge -\log\mu \} \Big). \]
In particular, $\Omega \setminus \mathcal{E}$ is pseudoconvex, and hence the projection $\pi_n\big(E(\partial\Omega)\big)$ of the envelope of holomorphy $E(\partial \Omega)$ of $\partial\Omega$ onto $\C^n$ is contained in $ \overline{\Omega}\setminus \mathcal{E}$.

It remains to show that $E(\partial\Omega)$ is single-sheeted and coincides with $\overline\Omega \setminus \mathcal{E}$. Observe that for every $a \in \R$ the set $\overline\Omega \cap \{ \varphi \ge a \}$ is compact and hence, since $\varphi_\nu \to \varphi$ uniformly on compact subsets of $\C^n \setminus \mathcal{E}$, for each $a \in \R$ we can choose a natural number $N(a) \in \N$ such that
\[ \Omega \cap \{\varphi > a\} \subset \Omega \cap \{\varphi_{N(a)} > a - 1 \} = \Omega \cap \big\{\abs{P_{N(a)}} > e^{2^{N(a)}(a-1)}\big\} \subset \Omega \cap \{\varphi > a-2\}.\]
Fix some $a \in \R$ and let $N \coloneqq N(a)$. Observe that $P_N$, being a polynomial, has only finitely many singular values $c_1, c_2, \ldots, c_k$ and let $S \coloneqq \bigcup_{j=1}^k \{P_N = c_j\}$ (indeed, using the explicit formula for $P_N$ stated after Lemma \ref{thm_defPn}, one can even see that $k = 1$ and $c_1 = 0$). Let now $f \in CR(\partial\Omega)$. Since $\Omega$ is strictly pseudoconvex, $f$ extends to a holomorphic function on some one-sided neighbourhood $U \subset \overline\Omega$ of $\partial\Omega$, which will be denoted by $f$ as well. 

Let $H \subset \C^n$ denote a complex two-dimensional affine subspace of $\C^n$ which is obtained by fixing $n-2$ of the coordinates $z_1, z_2, \ldots, z_{n-1}, w$ (for $n=2$ the only possible choice is $H = \C^2$).
Then $\Omega \cap H = \bigcup_\alpha \Gamma_\alpha$ is the disjoint union of a family $\{\Gamma_\alpha\}$ of strictly pseudoconvex domains $\Gamma_\alpha \subset H \cong \C^2$, and $\partial_H \Gamma_\alpha \subset \partial \Omega \cap H$ for each $\alpha$, where $\partial_H \Gamma_\alpha$ denotes the boundary of $\Gamma_\alpha$ with respect to the relative topology on $H$. In particular, we can view each $\Gamma_\alpha$ as a strictly pseudoconvex domain in $\C^2$, and for each $\alpha$ the restriction of $f$ to $U \cap H$ defines a holomorphic function in a one-sided neighbourhood of $\partial_H \Gamma_\alpha$. With the situation reduced to a two-dimensional case, we can now argue as in the example from introduction and conclude from Theorem A in $[J]$ that $E(\partial_H \Gamma_\alpha)$ is single-sheeted (of course here $E(\partial_H \Gamma_\alpha)$ denotes the envelope of holomorphy of $\partial_H\Gamma_\alpha$ with respect to functions holomorphic in $H \cong \C^2$). On the other hand, since for each $\nu \in \N$ the restriction $P_\nu|_H$ is again a polynomial and we can assume it to be nonconstant (for $\nu \ge \nu_0$ large enough this clearly is satisfied), for each $a' \in \R$ the sets $\{P_{N(a')}|_H = c\}$ with $c \in \C$, $\abs{c} > e^{2^{N(a')}(a'-1)}$, constitute a continuous family of analytic curves in $H \cong \C^2$, that fills $(\Omega \cap H) \cap \{\varphi > a'\}$. Using the Kontinuit\"atssatz, we thus conclude that $E(\partial_H\Gamma_\alpha) = \overline{\Gamma}_\alpha \cap \{\varphi > -\infty\} = \overline{\Gamma}_\alpha \setminus \mathcal{E}$ for each $\alpha$. Hence, since the domains $\Gamma_\alpha$ are disjoint and pseudoconvex, we get that $E(\bigcup_\alpha \partial_H\Gamma_\alpha)$ is single-sheeted and $(\Omega \cap H) \setminus \mathcal{E} \subset E(\bigcup_\alpha \partial_H\Gamma_\alpha)$. In particular, $f|_{U \cap H}$ extends to a holomorphic function
\[ f_H \colon (\Omega \setminus \mathcal{E}) \cap H \to \C, \qquad f_H = f \text{ near } \partial \Omega.\]
Observe that this already proves our claim in the case $n = 2$.

Assume now that $n \ge 3$. For each $c \in \C \setminus \{c_1, c_2, \ldots, c_k\}$ the hypersurface $\{P_N = c\}$ is a Stein manifold of dimension at least 2, and if $\abs{c} > e^{{2^N}(a-1)}$, then each connected component of $\,\Omega_c \coloneqq \Omega \cap \{P_N = c\}$ is a bounded strictly pseudoconvex domain in $\{P_N = c\}$. Further, $f$ restricts to a holomorphic function on $\Omega_c \setminus K$, where $K \subset \Omega_c$ is a compact set of the form $K = \Omega_c \setminus \tilde{U}$ for some one-sided neighbourhood $\tilde{U} \subset U$ of $\partial\Omega$. Since each connected component $\Gamma$ of $\Omega_c$ is bounded and strictly pseudoconvex, the boundary of $\Gamma$ in $\{P_N = c\}$ is connected, and hence we can assume $\Gamma \setminus K = \Gamma \cap \tilde{U}$ to be connected. Thus we can apply Hartogs theorem on removability of compact singularities to extend $f|_{\Omega_c \setminus K}$ to a holomorphic function $\tilde{f}_c$ on $\Omega_c$ (for a version of the classical Hartogs theorem in the setting of Stein manifolds see Corollary 4.2 in [AH]). In this way we can define a function
\[ f_a \colon \big[ \Omega \cap \{P_N > e^{2^N(a-1)}\} \big] \setminus S \to \C, \qquad f_a = f \text{ near } \partial \Omega, \]
by letting $f_a(z,w) = \tilde{f}_c(z,w)$ if $P_N(z,w) = c$. We claim that for every two-dimensional subspace $H \subset \C^n$ described above, the functions $f_a$ and $f_H$ coincide on their common domain of definition, namely, on the set $\big[\Omega \cap H \cap \{P_N > e^{2^N(a-1)}\}\big] \setminus S$. Indeed, let $c \in \C \setminus \{c_1, c_2, \ldots, c_k\}$, $\abs{c} > e^{2^N(a-1)}$. Since the restriction $P_N|_H$ is again a (nonconstant) polynomial, the set $\gamma_c \coloneqq \Omega \cap H \cap \{P_N = c\}$ is an analytic curve in $\Omega \cap H \cap \{P_N > e^{2^N(a-1)}\}$. Observe that the boundary of $\gamma_c$ is contained in $\partial \Omega$ and recall that $f_a$ and $f_H$ are holomorphic on $\gamma_c$ and coincide near $\partial \Omega$. Thus it follows from the uniqueness theorem that $f_a = f_H$ on $\gamma_c$. Hence, since $c \in \C \setminus \{c_1, c_2, \ldots, c_k\}$ with $\abs{c} > e^{2^N(a-1)}$ was arbitrary, we conclude that 
\begin{equation} \label{equ_fafH}
f_a = f_H \quad\text{on} \quad \big[\Omega \cap H \cap \{P_N > e^{2^N(a-1)}\}\big] \setminus S.
\end{equation}
In particular, this shows that $f_a$ is holomorphic in each variable separately (recall the definition of $H$). Thus by Hartogs separate analyticity theorem, $f_a$ is a holomorphic function on $\big[\Omega \cap \{P_N > e^{2^N(a-1)}\}\big] \setminus S$. Moreover, we see from $(\ref{equ_fafH})$ and the holomorphicity of $f_H$ on $(\Omega \cap H) \setminus \mathcal{E} \supset \Omega \cap H \cap \{ P_N > e^{2^N(a-1)}\}$ that $f_a$ remains bounded near $S$. It follows then from Riemann's removable singularities theorem that $f_a$ extends to a holomorphic function $\tilde{f}_a$ on $\Omega \cap \{ P_N > e^{2^N(a-1)}\} \supset \Omega \cap \{\varphi > a\}$. Since $a \in \R$ was arbitrary, and since $\Omega \setminus \mathcal{E} = \bigcup_{a\in \R} \Omega \cap \{\varphi > a\}$, we conclude that $f$ has a single-valued holomorphic extension to $\overline{\Omega} \setminus \mathcal{E}$. Hence $E(\partial\Omega)$ is single-sheeted and $E(\partial\Omega) = \overline{\Omega} \setminus \mathcal{E}$.

Now we can construct a smooth $CR$ function $f$ on $\partial\Omega$ which extends inside $\Omega$ exactly to $\overline{\Omega} \setminus \mathcal{E}$. In order to do so, let 
\[ \widetilde\Omega \coloneqq \big\{ (z,w) \in \C^n : \varphi(z,w) + \big( \norm{z}^2 + \abs{w}^2 \big) < C_2 \big\}, \]
where the constant $C_2>C_1$. Then the domain $\widetilde\Omega$ is also pseudoconvex and $\overline\Omega\subset\widetilde\Omega$. As before, we see that $\widetilde\Omega \setminus \mathcal{E}$ is pseudoconvex; hence there exists a holomorphic function $f \in \mathcal{O}(\widetilde\Omega\setminus \mathcal{E})$ which does not extend to $\mathcal{E}$. Then $f\vert_{\partial\Omega}$ is a function as required. \hfill $\Box$ \\[0.5cm]

Finally we state some open questions related to the content of the paper (and also related to each other).\bigskip

{\bf Question 1.} Let $\Omega\subset\C^n$, $n\ge 2$, be an unbounded strictly pseudoconvex domain. For each $R>0$, consider the hull
$\widehat{\partial\Omega\cap\Bbarn(0,R)}^{\raisebox{-1ex}{\scriptsize$A(\Omega)$}}$ of the set $\partial\Omega\,\cap\,\Bbarn(0,R)$ with respect to the algebra $\mathcal A(\Omega)$ of the functions holomorphic in $\Omega$, which are continuous up to the boundary $\partial \Omega$. Is it true that $\bigcup_{R>0}\widehat{\partial\Omega \cap \Bbarn(0,R)}^{\raisebox{-1ex}{\scriptsize$A(\Omega)$}}=\overline\Omega$?\bigskip

{\bf Question 2.} Is it true that there exist a properly embedded into $\C^n$, $n\ge 2$, smooth Levi-flat hypersurface $\mathcal M$  and an unbounded strictly pseudoconvex domain $\Omega\subset\C^n$ such that $\mathcal M\subset\Omega$?\bigskip

{\bf Question 3.} Let $\Omega \subset \C^n$, $n \ge 2$, be an unbounded strictly pseudoconvex domain. Does it follow that its boundary $\partial\Omega$ is connected?
\vspace{1truecm}

\noindent \textbf{Remark.} After submitting this paper to arXiv the authors were informed by M. Brunella that the answers to Questions 1 and 3 are negative and to Question 2 is positive.

To show this, M. Brunella suggests to consider a domain $W \subset \C^2_{z,w}$ biholomorphic to $\C^2_{z,w}$ such that $\{z \in \C_z : (z,0) \in W\} = \bigcup_{k=1}^N D_k$ and $\{z \in \C_z : (z,0) \in \overline{W}\} = \bigcup_{k=1}^N \overline{D}_k$, where $D_1, D_2, \ldots, D_N$ are bounded domains in $\C_z$ with $C^1$-smooth boundaries such that $\overline{D}_1, \overline{D}_2, \ldots, \overline{D}_N$ are pairwise disjoint. The existence of such a domain is granted by Corollary 1.1 in [G]. Let now $\delta > 0$ be so small that the $\delta$-neighbourhoods $U_k^\delta$ of $D_k$ in $\C_z$, $k = 1, 2, \ldots, N$, are still pairwise disjoint, and for each $k = 1, 2, \ldots, N$ consider a strictly subharmonic function $\varphi_k^\delta \in C^\infty(\overline{U_k^\delta})$ such that $\partial_{W} \big[W \cap \{(z,w) \in \C^2 : z \in U_k^\delta, \abs{w} \le e^{\varphi_k^\delta(z)} \}\big] \subset \{(z,w) \in \C^2 : z \in U_k^\delta, \abs{w} = e^{\varphi_k^\delta(z)} \}$, where $\partial_{W}$ denotes the boundary operator in the relative topology of $W$, and such that the set $\{(z,w) \in \C^2 : z \in \partial U_k^\delta, \abs{w} = e^{\varphi_k^\delta(z)} \}$ is disjoint from $\overline{W}$. Observe that these conditions are satisfied for $\delta$ small enough and $\varphi_k^\delta$ close enough to $-\infty$ due to the above property of $\overline{W}$. Fix now such $\delta$ and $\varphi_k^\delta$, $k = 1, 2, \ldots, N$, and consider an unbounded connected component $\widetilde{W}$ of the set $W \setminus \bigcup_{k=1}^N \{(z,w) \in \C^2 : z \in U_k^\delta, \abs{w} \le e^{\varphi_k^\delta(z)}\}$. Then by construction, $\widetilde{W}$ is strictly pseudoconvex along $\partial \widetilde{W} \setminus \partial W$ and, moreover, $\partial \widetilde{W} \setminus \partial W$ has at least $N$ different connected components. Let $F$ be a biholomorphic map of $W$ to $\C^2$ and define the domain $\Omega$ by $\Omega \coloneqq F(\widetilde{W})$. Then, since $\partial \Omega = F(\partial \widetilde{W} \setminus \partial W)$, $\Omega$ is an unbounded strictly pseudoconvex domain in $\C^2$ with at least $N$ boundary components, which gives a negative answer to Question 3. This domain contains a properly embedded into $\C^2$ Levi-flat hypersurface, namely, the surface $F\big((\partial D \times \C_w) \cap W\big)$, where $D$ is any open disc in $\pi_z(W) \setminus \bigcup_{k=1}^N U_k^\delta$. This gives a positive answer to Question 2. Finally let $\Omega' \subset \C^2$ be a strictly pseudoconvex domain such that $\overline{\Omega} \subset \Omega'$ and $F^{-1}(\Omega') \cap \{w=0\} = \varnothing$ (such a domain $\Omega'$ can be obtained, for example, by repeating the construction of $\Omega$ with $\varphi_k^\delta$ replaced by $\varphi_k^\delta - 1$, $k = 1, 2, \ldots, N$). Then $\phi \coloneqq (\log \abs{w}) \circ F^{-1}$ is a continuous plurisubharmonic function on $\Omega'$, hence $\Omega_c' \coloneqq \{(z,w) \in \Omega': \phi(z,w) < c\}$ is Runge in $\Omega'$ for every $c \in \R$ (see Corollary 1 of \S 4 in [N]). Moreover, by construction of $\Omega$, for suitably chosen $c$ the set $\Omega_c'$ is a neighbourhood of $\partial \Omega$ and $\overline{\Omega} \setminus \Omega_c' \neq \varnothing$. After fixing such $c$ we then conclude that $\bigcup_{R>0} \widehat{\partial \Omega \cap \Bbarn(0,R)}^{\raisebox{-1ex}{\scriptsize$A(\Omega)$}} \subset \bigcup_{R>0} \widehat{\partial \Omega \cap \Bbarn(0,R)}^{\raisebox{-1ex}{\scriptsize$\mathcal{O}(\Omega')$}} \cap \overline{\Omega} \subset \Omega_c' \cap \overline{\Omega} \subsetneq \overline{\Omega}$ which gives a negative answer to Question 1.

%
 \vspace{1truecm}

%
%
%
  {\sc T. Harz: Department of Mathematics, University of Wuppertal --- 42119 Wuppertal, Germany}
  
  {\em e-mail address}: {\texttt harz@math.uni-wuppertal.de}\vspace{0.3cm}
  
  {\sc N. Shcherbina: Department of Mathematics, University of Wuppertal --- 42119 Wuppertal, Germany}
  
  {\em e-mail address}: {\texttt shcherbina@math.uni-wuppertal.de}\vspace{0.3cm}
  
  {\sc G. Tomassini: Scuola Normale Superiore, Piazza dei Cavalieri, 7 --- 56126 Pisa, Italy}
  
  {\em e-mail address}: {\texttt g.tomassini@sns.it}
\end{document}